\title{Random homomorphisms into the orthogonality graph}
\author{D\'avid Kunszenti-Kov\'acs\footnote{Research supported by ERC
Consolidator Grant 648017}, L\'aszl\'o Lov\'asz\footnote{Research supported by
ERC Synergy Grant No.~810115.} ~and Bal\'azs Szegedy\footnote{Research was
partially supported by the NKFIH "\'Elvonal" KKP 133921 grant.}\\
Alfr\'ed R\'enyi Institute of Mathematics\\
Budapest, Hungary}
\date{\today}

\documentclass[11pt,fleqn]{article}
\usepackage{amsmath,amssymb,graphicx,bbm,bm,delarray,pict2e,ntheorem}
\usepackage[hypertex]{hyperref}
\usepackage[normalem]{ulem}
\usepackage{srcltx}
\long\def\ignore#1{}

\begin{document}

\newtheorem{theorem}{Theorem}
\newtheorem{prop}[theorem]{Proposition}
\newtheorem{lemma}[theorem]{Lemma}
\newtheorem{claim}{Claim}
\newtheorem{corollary}[theorem]{Corollary}
\theorembodyfont{\rmfamily}
\newtheorem{remark}[theorem]{Remark}
\newtheorem{example}{Example}
\newtheorem{conj}{Conjecture}
\newtheorem{problem}[theorem]{Problem}
\newtheorem{step}{Step}
\newtheorem{alg}{Algorithm}
\newenvironment{proof}{\medskip\noindent{\bf Proof. }}{\hfill$\square$\medskip}
\newenvironment{proof*}[1]{\medskip\noindent{\bf Proof of #1.}}{\hfill$\square$\medskip}
\def\mb{\mathbb}
\def\mr{\mathrm}
\def\mc{\mathcal}
\def\ms{\mathscr}
\def\mk{\mathfrak}
\def\mf{\mathbf}
\def\R{\mathbb{R}}
\def\one{\mathbbm1}
\def\T{^{\sf T}}
\def\Pr{{\sf P}}
\def\E{{\sf E}}
\def\Q{{\mathbf Q}}
\def\bd{\text{bd}}
\def\eps{\varepsilon}
\def\wh{\widehat}
\def\cork{\text{\rm corank}}
\def\rank{\text{\rm rank}}
\def\Ker{\text{\rm Ker}}
\def\rk{\text{\rm rank}}
\def\supp{\text{\rm supp}}
\def\diag{\text{\rm diag}}
\def\sep{\text{\rm sep}}
\def\tr{\text{\rm tr}}
\def\vol{\text{\rm vol}}
\def\gen{\text{\rm gen}}
\def\Det{\text{\rm Det}}
\def\lin{\text{\rm lin}}
\def\intl{\int\limits}
\def\Hom{\text{\rm Hom}}
\def\Ort{\text{\sf O}}

\def\AA{\mathcal{A}}\def\BB{\mathcal{B}}\def\CC{\mathcal{C}}
\def\DD{\mathcal{D}}\def\EE{\mathcal{E}}\def\FF{\mathcal{F}}
\def\GG{\mathcal{G}}\def\HH{\mathcal{H}}\def\II{\mathcal{I}}
\def\JJ{\mathcal{J}}\def\KK{\mathcal{K}}\def\LL{\mathcal{L}}
\def\MM{\mathcal{M}}\def\NN{\mathcal{N}}\def\OO{\mathcal{O}}
\def\PP{\mathcal{P}}\def\QQ{\mathcal{Q}}\def\RR{\mathcal{R}}
\def\SS{\mathcal{S}}\def\TT{\mathcal{T}}\def\UU{\mathcal{U}}
\def\VV{\mathcal{V}}\def\WW{\mathcal{W}}\def\XX{\mathcal{X}}
\def\YY{\mathcal{Y}}\def\ZZ{\mathcal{Z}}

\def\Ab{\mathbf{A}}\def\Bb{\mathbf{B}}\def\Cb{\mathbf{C}}
\def\Db{\mathbf{D}}\def\Eb{\mathbf{E}}\def\Fb{\mathbf{F}}
\def\Gb{\mathbf{G}}\def\Hb{\mathbf{H}}\def\Ib{\mathbf{I}}
\def\Jb{\mathbf{J}}\def\Kb{\mathbf{K}}\def\Lb{\mathbf{L}}
\def\Mb{\mathbf{M}}\def\Nb{\mathbf{N}}\def\Ob{\mathbf{O}}
\def\Pb{\mathbf{P}}\def\Qb{\mathbf{Q}}\def\Rb{\mathbf{R}}
\def\Sb{\mathbf{S}}\def\Tb{\mathbf{T}}\def\Ub{\mathbf{U}}
\def\Vb{\mathbf{V}}\def\Wb{\mathbf{W}}\def\Xb{\mathbf{X}}
\def\Yb{\mathbf{Y}}\def\Zb{\mathbf{Z}}

\def\ab{\mathbf{a}}\def\bb{\mathbf{b}}\def\cb{\mathbf{c}}
\def\db{\mathbf{d}}\def\eb{\mathbf{e}}\def\fb{\mathbf{f}}
\def\gb{\mathbf{g}}\def\hb{\mathbf{h}}\def\ib{\mathbf{i}}
\def\jb{\mathbf{j}}\def\kb{\mathbf{k}}\def\lb{\mathbf{l}}
\def\nb{\mathbf{n}}\def\ob{\mathbf{o}}
\def\pb{\mathbf{p}}\def\qb{\mathbf{q}}\def\rb{\mathbf{r}}
\def\sb{\mathbf{s}}\def\tb{\mathbf{t}}\def\ub{\mathbf{u}}
\def\vb{\mathbf{v}}\def\wb{\mathbf{w}}\def\xb{\mathbf{x}}
\def\yb{\mathbf{y}}\def\zb{\mathbf{z}}

\def\Abb{\mathbb{A}}\def\Bbb{\mathbb{B}}\def\Cbb{\mathbb{C}}
\def\Dbb{\mathbb{D}}\def\Ebb{\mathbb{E}}\def\Fbb{\mathbb{F}}
\def\Gbb{\mathbb{G}}\def\Hbb{\mathbb{H}}\def\Ibb{\mathbb{I}}
\def\Jbb{\mathbb{J}}\def\Kbb{\mathbb{K}}\def\Lbb{\mathbb{L}}
\def\Mbb{\mathbb{M}}\def\Nbb{\mathbb{N}}\def\Obb{\mathbb{O}}
\def\Pbb{\mathbb{P}}\def\Qbb{\mathbb{Q}}\def\Rbb{\mathbb{R}}
\def\Sbb{\mathbb{S}}\def\Tbb{\mathbb{T}}\def\Ubb{\mathbb{U}}
\def\Vbb{\mathbb{V}}\def\Wbb{\mathbb{W}}\def\Xbb{\mathbb{X}}
\def\Ybb{\mathbb{Y}}\def\Zbb{\mathbb{Z}}

\def\Af{\mathfrak{A}}\def\Bf{\mathfrak{B}}\def\Cf{\mathfrak{C}}
\def\Df{\mathfrak{D}}\def\Ef{\mathfrak{E}}\def\Ff{\mathfrak{F}}
\def\Gf{\mathfrak{G}}\def\Hf{\mathfrak{H}}\def\If{\mathfrak{I}}
\def\Jf{\mathfrak{J}}\def\Kf{\mathfrak{K}}\def\Lf{\mathfrak{L}}
\def\Mf{\mathfrak{M}}\def\Nf{\mathfrak{N}}\def\Of{\mathfrak{O}}
\def\Pf{\mathfrak{P}}\def\Qf{\mathfrak{Q}}\def\Rf{\mathfrak{R}}
\def\Sf{\mathfrak{S}}\def\Tf{\mathfrak{T}}\def\Uf{\mathfrak{U}}
\def\Vf{\mathfrak{V}}\def\Wf{\mathfrak{W}}\def\Xf{\mathfrak{X}}
\def\Yf{\mathfrak{Y}}\def\Zf{\mathfrak{Z}}

\maketitle

\tableofcontents

\begin{abstract}
Subgraph densities have been defined, and served as basic tools, both in the
case of graphons (limits of dense graph sequences) and graphings (limits of
bounded-degree graph sequences). While limit objects have been described for
the "middle ranges", the notion of subgraph densities in these limit objects
remains elusive. We define subgraph densities in the orthogonality graphs on
the unit spheres in dimension $d$, under appropriate sparsity condition on the
subgraphs. These orthogonality graphs exhibit the main difficulties of defining
subgraphs the ``middle'' range, and so we expect their study to serve as a key
example to defining subgraph densities in more general Markov spaces.

The problem can also be formulated as defining and computing random orthogonal
representations of graphs. Orthogonal representations have played a role in
information theory, optimization, rigidity theory and quantum physics, so to
study random ones may be of interest from the point of view of these
applications as well.
\end{abstract}

\section{Introduction}

Let $H_d$ denote the orthogonality graph on $S^{d-1}$, i.e., the infinite graph
whose node set is the unit sphere $S^{d-1}$, and two nodes are adjacent if they
are orthogonal (as vectors in $\R^d$). For a finite graph $G$, we call a
homomorphism of $G$ into $H$ an {\it ortho-homomorphism} of $G$ (in dimension
$d$).

Our motivation for studying ortho-homomorphisms comes from graph limit theory.
This theory is rather well worked out for dense graphs on one end of scale
(where the limit objects are graphons), and bounded degree graphs on the other
(where the limit objects are graphings). In spite of several efforts to extend
the theory to the intermediate cases, no complete theory has been developed.

One basic question is: what structures can serve as limit objects for
``convergent'' graph sequences? Here at least we seem to have a common ground:
symmetric probability measures on the unit square (or on any other standard
probability space; these measures are essentially equivalent to time-reversible
Markov chains with a stationary distribution). These structures, which we call
{\it Markov spaces}, capture most special cases of interest, including limit
objects for $L_p$-convergence \cite{BCCZ}, shape convergence \cite{KLSz} and
action convergence \cite{BackSz}.

However, all these limit notions are defined through a global (right)
convergence. To characterize them by local (left) convergence, we need to
define the density of subgraphs in Markov spaces. At this time, we have a
definition beyond the the two extreme cases in rather special cases only.

Our main goal in this paper is to define subgraph densities in the
orthogonality graphs $H_d$ (which have a natural Markov space associated with
them). These spaces exhibit the main difficulties of the ``middle'' range, and
so we expect their study to serve as a key example to defining subgraph
densities in more general Markov spaces.

To justify this special choice, let us describe a somewhat unexpected further
connection. An ortho-homomorphism of $G$ in dimension $d$ is the same thing as
an orthonormal representation of the complementary graph $\overline{G}$ (see
\cite{GeomBook}). Such representations have played a role in information theory
\cite{LL79}, graph algorithms \cite{GLS1,GLSBook}, rigidity of frameworks
\cite{Alfa}, and quantum physics \cite{CSWquant}. Our results in this paper
could be thought of as establishing further connections with probability and
measure theory.

A related question is to define a {\it random} homomorphism of $G$ into $H_d$.
The notion of a random edge (the uniform distribution on orthogonal pairs of
vectors) is trivial, but for more complicated graphs, it is not obvious what
``random'' should mean. Ortho-homomorphisms from a given graph form a real
algebraic variety $\Hom(G,d)$, which can have a very complicated topology; but
ortho-homomorphisms in general position (see below) form a smooth semialgebraic
variety $\Sigma_{G,d}$. We could consider the surface measure inherited from
the ambient space $(\R^d)^V$; however, this does not seem to have really useful
properties. Natural conditions to impose are invariance under orthogonal
transformations of $\R^d$ and the {\it Markov property} (see Section
\ref{SEC:MARKOV}).

The example of the 4-cycle in dimension $3$ should be a warning. Obviously, for
every homomorphism $C_4\to H_3$, one pair of nonadjacent nodes will be mapped
onto parallel vectors (the other pair can form an arbitrary angle). But which
one? The variety $\Hom(G,d)$ splits into two, and $\Sigma_{C_4,3}=\emptyset$.

In this paper we show that for several classes of graphs satisfying appropriate
sparsity conditions, a measure on their ortho-homomorphisms in a given
dimension $d$ can be defined, with good properties. The measure we define is
always a Radon measure, but finiteness is not guaranteed. Indeed, we'll give
examples where this measure is finite, and so it can be scaled to a probability
measure (defining a ``random ortho-homomorphism''); unfortunately, we also have
examples where the measure is infinite. The combinatorial significance of this
finiteness (depending on $G$ and $d$) remains an interesting unsolved problem.
When this measure is finite, then its value on the set of all
ortho-homomorphisms appears to be good substitute for the homomorphism density.

We describe three methods for defining subgraph densities in $H_d$.

\medskip

\noindent{\it Sequential mapping.} One of our constructions works for graphs
not containing a complete bipartite graph $K_{a,b}$ with $a+b>d$. This
condition is equivalent to saying that $\overline{G}$ is $(n-d)$-connected.
We'll call such graphs {\it $d$-sparse}. It implies, in particular, that every
node has degree at most $d-1$. We say that a mapping $x:~V\to \R^d$ is in {\it
general position}, if any $d$ elements of $V$ are mapped onto linearly
independent vectors. The following fact was proved in \cite{LSS} (Theorem 2.1).

\begin{prop}\label{PROP:LSS}
A graph $G$ has an ortho-homomorphism in $\R_d$ in general position if and only
if it is $d$-sparse.
\end{prop}

The main tool in the proof of Proposition \ref{PROP:LSS} was the following. Let
us order the nodes of $G$ in some way, and choose the images of the nodes
one-by-one. At every step, the new node is restricted to unit vectors
orthogonal to those neighbors that are already mapped. By the degree condition,
the available vectors form a nonempty sphere of some dimension, and we choose a
next vector on this sphere randomly and uniformly. Repeating this for all
nodes, we get an ortho-homomorphism, which we call a {\it random sequential
ortho-homomorphism} of $G$. The fact that this ortho-homomorphism is in general
position almost surely is the main result in \cite{LSS}.

The distribution of the random sequential ortho-homomorphism may depend on the
ordering of the nodes. If $G$ is a tree, then we get the same distribution for
every search order $(1,\dots,n)$ of the nodes, but not for other orders.
However, we can define a density function for which the modified distribution
will be independent of the ordering. One of our main results can be stated as
follows:

\begin{theorem}\label{THM:ORTHO-MARKOV1}
For every simple $d$-sparse graph $G$, there exists a nonzero Radon measure on
ortho-homomorphisms in dimension $d$ with a Markovian conditioning.
\end{theorem}

The measure of all homomorphisms is a good generalization of the notion of
homomorphism density, a basic tool in the theory of dense graph limits. The
Markov property is usually defined for probability measures, and we cannot
always normalize our measure on ortho-homomorphisms to a probability measure.
We'll describe the formal definition later.

\medskip

\noindent{\it Spectral methods.} Our other construction is based on functional
analysis. The orthogonality graph $H_d$ defines a compact linear operator
$\Ab_d:~L^2(S^{d-1},\pi)\to L^2(S^{d-1},\pi)$, where $\pi$ is the uniform
probability measure on $S^{d-1}$, and $(\Ab_d f)(x)$ is the average of $f$ on
the $(d-2)$-dimensional sphere orthogonal to $x$. Taking the $k$-th power of
this operator corresponds to subdividing each edge of $G$ by $k-1$ nodes. It
turns out that the square of this operator is smooth enough so that random
subgraphs and subgraph densities can be defined by ``classical'' formulas.
Also, the trace of $\Ab_d^k$ gives the density of $k$-cycles (at least for
sufficiently large $k$).

Using the spectral decomposition of $\Ab_d$, we derive explicit formulas for
the densities of cycles in $H_d$. As an interesting fact, cycle densities in
$H_4$ can be expressed by the zeta-function.

\medskip

\noindent{\it Approximation by graphs and graphons.} The third method of
defining and calculating subgraph densities in $H_d$ is based on approximating
$H_d$ by graphons and finite graphs, and calculating the density in $H_d$ as
the limit of densities in these approximations.

A consequence of our results is that $H_d$ {\it is} the limit of finite graphs
in the left-convergence sense. While this property is easy for graphons, it is
not known in the bounded-degree case whether all graphings can be approximated
by finite graphs (this is equivalent for the famous soficity problem for
finitely generated groups). So the fact that $H_d$ is ``sofic'' in this sense
has some independent interest.

\medskip

Finally, it should be noted that a good part of the results of this paper
extend to more general Markov spaces. In particular, a general version of the
operator $\Ab_d$ is called a {\it graphop} and it arises in the theory of
action convergence \cite{BackSz} and, in an equivalent form, in the theory of
shape convergence \cite{KLSz}.

\section{Preliminaries}

\subsection{Notation}

We consider the unit sphere $S^{d-1}$ in $\R^d$. (The cases $d\le 2$ are very
simple, so to avoid trivial complications, we assume throughout that $d\ge 3$.)
For two real quantities (depending on a choice of points in $S^{d-1}$), let
$A\lesssim B$ denote that there is a constant $c>0$ such that $A \le cB$. Here
the constant may depend on the dimension and on the graph denoted by $G$, but
not on other variables. Let $A_k$ denote the surface area of $S^k$. It is well
known that
\begin{equation}\label{EQ:SK-DEF}
A_k=
  \begin{cases}
    \displaystyle\frac{2 (2\pi)^{k/2}}{(k-1)!!} & \text{if $k$ is even}, \\[12pt]
    \displaystyle\frac{(2\pi)^{(k+1)/2}}{(k-1)!!} & \text{if $k$ is odd},
  \end{cases}
\end{equation}
and for $a,b\in \Zbb_+$,
\begin{equation}\label{EQ:AST-DEF}
\intl_0^{\pi/2} (\sin \theta)^a (\cos \theta)^b\,d\theta
= \Big(\frac{\pi}{2}\Big)^{e(a,b)} \frac{(a-1)!!(b-1)!!}{(a+b)!!} = \frac{A_{a+b+1}}{A_a A_b},
\end{equation}
where $e(a,b)=(a-1)(b-1) \mod 2$, and $(-1)!! = 0!! =1!!=1$.

When we talk about a ``random'' point of a sphere, we mean a random point from
the uniform distribution on the sphere.

\subsection{Generalized determinants}\label{SEC:DETS}

For a finite set $X=\{x_1,\dots,x_m\}\subseteq\R^d$, we define the quantity
\[
\Det(X)=\Det(x_1,\dots,x_m)=|x_1\land\dots\land x_m|=\sqrt{\det\big((x_i\T x_j)_{i,j=1}^m\big)}.
\]
We define $\Det(\emptyset)=1$. For $m=1$, $\Det(X)=\Det(x_1)=|x_1|$. Note that
$\Det(X)\ge 0$, and $\Det(X)>0$ if and only if $X$ consists of linearly
independent vectors.

\begin{lemma}\label{LEM:XEXPECT}
Let $n,d\in\Nbb$ and $p\in\R$ such that $1\le n< d$, and let $x_1,\dots,x_n$ be
independent random points on $S^{d-1}$. Then
\[
\E\Big(\frac{\Det(x_1,\dots,x_n)^p}{\Det(x_1,\dots,x_{n-1})^p}\Big) \qquad\text{and}\qquad
\E\big(\Det(x_1,\dots,x_n)^p\big)
\]
are finite if and only if $p>n-d-1$. If $p$ is an integer, then we have the
explicit formulas
\begin{align*}
\E\Big(\frac{\Det(x_1,\dots,x_n)^p}{\Det(x_1,\dots,x_{n-1})^p}\Big)
= \frac{A_{d+p-1} A_{d-n}}{A_{d-1}A_{d-n+p}}.
\end{align*}
and
\begin{align*}
\E\big(\Det(x_1,\dots,x_n)^p\big) = \Big(\frac{A_{d+p-1}}{A_{d-1}}\Big)^{n-1}
\frac{A_{d-2}\cdots A_{d-n}}{A_{d+p-2}\cdots A_{d+p-n}}.
\end{align*}
\end{lemma}

In these expectations, we could condition on (say) a fixed $x_1$, by the
symmetry of the sphere. Note that $p$ may be negative, but if $p\le n-d-1$,
then the expectations are infinite.

\begin{proof}
For $n=1$ the identities are trivial, so we assume that $n\ge 2$. The ratio
$\Det(x_1,\dots,x_n)/\Det(x_1,\dots,x_{n-1})$ is the (unsigned) distance of
$x_n$ from the subspace $L=\lin(x_1,\dots,x_{n-1})$, which has dimension $n-1$
with probability $1$. The distribution of this distance is independent of
$x_1,\dots,x_{n-1}$, so we may fix $L$ and just take expectation in $x_n$.

Let $\theta$ be the angle between $x_n$ and $L$ ($0\le\theta\le\pi/2$), then
\[
\frac{\Det(x_1,\dots,x_n)}{\Det(x_1,\dots,x_{n-1})}= \sin\theta.
\]
For a fixed $\theta$, points at this distance from $L$ form the direct product
of the two spheres $L\cap (\cos\theta)S^{d-1}$ and $L^\perp \cap
(\sin\theta)S^{d-1}$, and so their density is proportional to
$(\cos\theta)^{n-2}(\sin\theta)^{d-n}$. Hence
\[
\E\Big(\frac{\Det(x_1,\dots,x_n)^p}{\Det(x_1,\dots,x_{n-1})^p}\Big)=
\frac{\intl_0^{\pi/2} (\sin\theta)^{d-n+p}(\cos\theta)^{n-2}\,d\theta}
{\intl_0^{\pi/2} (\sin\theta)^{d-n}(\cos\theta)^{n-2}\,d\theta}.
\]
Using that $2\theta/\pi\le\sin\theta\le \theta$, it follows that the numerator
is finite if and only if $d_n+p>-1$, proving the first assertion. Substituting
from \eqref{EQ:AST-DEF} for integral $p$, we get the first formula in the
lemma.

To prove the second identity, we use the telescopic product decomposition
\[
\Det(x_1,\dots,x_n)^p = \prod_{r=2}^n \frac{\Det(x_1,\dots,x_r)^p}{\Det(x_1,\dots,x_{r-1})^p}.
\]
As remarked above, the factors are independent random variables, and hence
\begin{align*}
\E(\Det(x_1,\dots,x_n)^p) &= \prod_{r=2}^n \E\Big(\frac{\Det(x_1,\dots,x_r)^p}{\Det(x_1,\dots,x_{r-1})^p}\Big)
= \prod_{r=2}^n  \frac{A_{d+p-1} A_{d-r}}{A_{d-1}A_{d-r+p}}\\
&=\Big(\frac{A_{d+p-1}}{A_{d-1}}\Big)^{n-1} \frac{A_{d-2}\cdots A_{d-n}}{A_{d+p-2}\cdots A_{d+p-n}}.
\end{align*}
\end{proof}

For small values of $|p|$, we can cancel most of the terms on the right hand
side of the second equality. The most important special case for us will be
$p=-1$:
\[
\E\Big(\frac1{\Det(x_1,\dots,x_n)}\Big) = \frac{A_{d-2}^n}{A_{d-1}^{n-1}A_{d-n-1}}.
\]
This identity makes sense for $n=0$ as well, and it is trivially valid. In
particular $1/\Det$ is integrable provided $n\le d-1$. We shall make use of the
following one-sided bound on averages of such inverses of determinants.

\begin{lemma}\label{LEM:DETAVERAGE}
Let $x_1,\dots,x_n\in S^{d-1}$, and let $B_r(x)$ denote the $r$-neighborhood of
$x$ on $S^{d-1}$. Let $D_r(x_1,\dots,x_n)$ denote the average of
$1/\Det(x_1,\dots,x_n)$ over $B_r(x_1)\times \dots\times B_r(x_n)$. Then
\begin{equation}\label{EQ:D-AVE}
D_r(x_1,\dots,x_n)\Det(x_1,\dots,x_n) < C_{n,d},
\end{equation}
where $C_{n,d}>0$ may depend on $d$ and $n$, but not on $r$ and $(x_1,\dots,x_n)$.
\end{lemma}

\begin{proof}
We shall fix $d$, and proceed by induction on $n$. The case $n=1$ is trivial,
as the determinant is constant 1. So let us now assume $C_{n-1,d}$ exists, and
show that $C_{n,d}$ exists as well ($1<n<d$). Since $1/\Det(x_1,\dots,x_n)$ is
positive, integrable and continuous outside of the null-set of its
singularities, the map
\[
(x_1,\ldots,x_n)\mapsto D_r(x_1,\dots,x_n)\Det(x_1,\dots,x_n)
\]
is continuous, with a maximum $M_r$, and also $r\mapsto M_r$ is continuous on
$(0,\infty)$. To show that $M_r$ is bounded above note that it is constant once
$r\geq \pi$, and so we only need to show that it remains bounded above on some
finite interval $(0,\varepsilon]$, where $\varepsilon$ may be chosen
arbitrarily small. Let $q:=1/(\sqrt[n]{1,5}-1)$, and set
$\varepsilon:=1/(10q)$.

We distinguish two cases based on the relative positions of the points. We may
assume without loss of generality that the minimal distance $R$ from an $x_j$
to the subspace generated by the other $n-1$ points is realized for $j=n$.

\smallskip

\underline{Case 1}: $R\geq qr$.

\smallskip

Note that $R$ is also a lower bound on any distance from one of the $x_j$'s to
any subspace generated by some selection of the other points. In particular, we
have that for any $J\subseteq [n]$, $ \left| \bigwedge_{j\in J} x_j \right|\leq
\Det(x_1,\ldots,x_n)/R^{n-|J|} $. Therefore for any $\rho\in(\mb{R}^d)^n$ with
$|\rho_k|<r$ for all $1\leq k\leq n$,
\begin{align*}
\Det(x_1+\rho_1,&\ldots,x_n+\rho_n)\geq~ \Det(x_1,\ldots,x_n)-\sum_{J\subsetneq [n]}
\left| \bigwedge_{j\in J} x_j  \bigwedge_{i\in [n]\setminus J} \rho_i
\right|\\
\geq& ~\Det(x_1,\ldots,x_n)-\sum_{J\subsetneq [n]} \Det(x_1,\ldots,x_n)\frac{r^{n-|J|}}{R^{n-|J|}}\\
\geq & ~\Det(x_1,\ldots,x_n)\left(1-\sum_{j=1}^{n} \frac{1}{q^j}\binom{n}{j}\right)\\
=&~ \Det(x_1,\ldots,x_n)\left(2-\Big(1+\frac{1}{q}\Big)^n\right) = \frac{\Det(x_1,\ldots,x_n)}{2}.
\end{align*}
Consequently $D_r(x_1,\dots,x_n)\Det(x_1,\dots,x_n)\leq 2$.

\smallskip

\underline{Case 2}: $R< qr$.

\smallskip

In this case $\Det(x_1,\ldots,x_{n-1},x_n)\leq qr\Det(x_1,\ldots,x_{n-1})$. Fix
any choice of linearly independent points $y_i\in S^{d-1}$ ($1\leq i\leq n-1$).
Then the set of points $z$ such that
$\Det(y_1,\ldots,y_{n-1},z)=t\Det(y_1,\ldots,y_{n-1})$ form a
$(d-n+1)$-dimensional sphere of radius $t$ around the $(n-2)$-dimensional
subspace $\lin(y_1,\ldots,y_{n-1})$. After intersecting with $S^{d-1}$, the
dimension of suitable $z$'s is reduced to $d-n>0$. Now $\int_0^r (t^{d-n})/t
\,dt= r^{d-n}/(d-n)$, and so for any $y\in S^{d-1}\cap
~\lin(x_1,\ldots,x_{n-1})$ we obtain that
\[
\intl_{B_r(y)} \frac1{\Det(y_1,\ldots,y_{n-1},z)}\,d\pi(z)\le
\frac{C'_{d,n}r^{d-n}r^{n-2}}{\Det(y_1,\ldots,y_{n-1})}
= \frac{C'_{d,n}r^{d-2}}{\Det(y_1,\ldots,y_{n-1})},
\]
where $C'_{d,n}$ does not depend on $r$ or $y$ (recall that $r\in(0,1/(10q)]$).
Consequently,
\begin{align}\label{eq:singular_avg}
\nonumber \mb{E}_{z\in B_r(y)}\left(\frac1{\Det(y_1,\ldots,y_{n-1},z)}\right)&~\le
\frac1{\pi(B_r(y))}\frac{C'_{d,n}r^{d-2}}{r^{d-1}\Det(y_1,\ldots,y_{n-1})}\\
&~= \frac{C''_{d,n}}{r\Det(y_1,\ldots,y_{n-1})}.
\end{align}
Also note that replacing $y$ by any point not on $\lin(y_1,\ldots,y_{n-1})$
will actually increase the expectation (the distribution of the values of $t$
within the $r$-neighborhood gets shifted away from 0). Since the set of points
$(y_1,\ldots,y_{n-1})\in\prod_{j=1}^{n-1}B(x_j,r)$ that are not a linearly
independent $(n-1)$-tuple is of measure zero, we have the following.
\begin{align*}
D_r(x_1,\dots,x_n)&\leq qr\, D_r(x_1,\dots,x_{n-1})\\
&=~qr\,\mb{E}_{y_j\in B_r(x_j)}  \mb{E}_{z\in B_r(x_n)}\left(\frac1{\Det(y_1,\ldots,y_{n-1},z)}\right)\\
&\leq~qr\,\mb{E}_{y_j\in B_r(x_j)}  \mb{E}_{z\in B_r(y_{n-1})}\left(\frac1{\Det(y_1,\ldots,y_{n-1},z)}\right)\\
&\le qr\,\mb{E}_{y_j\in B_r(x_j)} \left(\frac{C''_{d,n}}{r\Det(y_1,\ldots,y_{n-1})}\right)\\
&=~ q\, \mb{E}_{y_j\in B_r(x_j)} \left(\frac{C''_{d,n}}{\Det(y_1,\ldots,y_{n-1})}\right)\\
&\le~q\;\frac{C_{n-1,d}\,C''_{d,n}}{\Det(x_1,\ldots,x_{n-1})}
\end{align*}
(we have used \eqref{eq:singular_avg} in the fourth step and the induction
hypothesis in the last). This implies the inequality in the lemma.
\end{proof}

The following lemma connects these reciprocals of determinants to the
orthogonality graph.

\begin{lemma}\label{LEM:N-LOOSE}
Let $0\le n<d$, and let $(x_1,\dots,x_n)$ be obtained by selecting a random
point $y$ on $S^{d-1}$, then selecting $n$ independent random points
$x_1,\dots,x_n$ from the ``equator'' $y^\perp\cap S^{d-1}$, then forgetting
$y$. Then the density function of $(x_1,\dots,x_n)$ is
\begin{equation}\label{EQ:SNSD}
s_{d,n}(x_1,\dots,x_n) = \frac{A_{d-1}^{n-1}A_{d-n-1}}{A_{d-2}^n}\, \frac1{\Det(x_1,\dots,x_n)}
\end{equation}
\end{lemma}

As remarked above in a different language (cf.~Lemma \ref{LEM:XEXPECT}),
$s_{d,n}\in L_p(S^{d-1},\pi)$ ($p\ge 1$) if and only if $p< d-n+1$.

\begin{proof}
Similarly as in the proof of Lemma \ref{LEM:XEXPECT}, we use induction on $n$.
For $n\le 1$ the assertion is trivial. Let $n\ge 2$, and let $L$ denote the
linear space spanned by $x_1,\dots,x_{n-1}$. With probability $1$,
$\dim(L)=n-1$. Clearly $L$ is uniformly distributed among all
$(n-1)$-dimensional subspaces of $y^\perp$, and so $y^\perp$ is uniformly
distributed among all linear hyperplanes containing $L$. So we construct $x_n$
by (a) choosing a random $(n-1)$-dimensional subspace $L$, (b) choosing a
random hyperplane $H$ containing $L$, and (c) choosing a random point from
$H\cap S^{d-1}$. Let us fix $L$, and let $\theta$ be the angle between $x_n$
and $L$. It is clear by symmetry that the density of $x_n$ depends only on
$\theta$.

For every choice of $H$, the distribution of $\theta$ is the same, and the
density of this distribution (in $[0,\pi]$ is proportional to
$(\cos\theta)^{n-2}(\sin\theta)^{d-n-1}$, as we have seen in the proof of Lemma
\ref{LEM:XEXPECT}. By the same argument, for a uniform random point $x'_n\in
S^{d-1}$, the angle $\theta'$ between $x_n'$ and $L$ has density proportional
to $(\cos\theta)^{n-2}(\sin\theta)^{d-n}$. It follows that the density of
$x_n$, relative to the uniform distribution on $S^{d-1}$, is proportional to
\[
\frac{(\cos\theta)^{n-2}(\sin\theta)^{d-n-1}}{(\cos\theta)^{n-2}(\sin\theta)^{d-n}}
=\frac1{\sin\theta} = \frac{\Det(x_1,\dots,x_{n-1})}{\Det(x_1,\dots,x_n)}.
\]
Since the density of $(x_1,\dots,x_{n-1})$ is proportional to
$1/\Det(x_1,\dots,x_{n-1})$ by the induction hypothesis, it follows that the
density of $(x_1,\dots,x_n)$ is proportional to
\[
\frac1{\Det(x_1,\dots,x_{n-1})}\,\frac{\Det(x_1,\dots,x_{n-1})}{\Det(x_1,\dots,x_n)} =
\frac1{\Det(x_1,\dots,x_n)}.
\]
The coefficient of proportionality can be computed by Lemma \ref{LEM:XEXPECT}.
\end{proof}

\subsection{Conditioning and Markov property}\label{SEC:MARKOV}

Let $V$ be a finite set, and let $(\Omega,\AA)$ be a measurable space (for most
of the paper, $\Omega=S^{d-1}$, and $\AA$ is the sigma-algebra of Borel sets).
Let $\Omega^{*V}$ denote the set of partial mappings $z:~ S\to J$, $S\subseteq
V$. Let $\varphi$ be a measure on $(\Omega^V,\AA^V)$, and let $\varphi^S$
denote the marginal of $\varphi$ on $(\Omega^S,\AA^S)$.

A family $(\varphi_z:~z\in \Omega^{*V})$ is a {\it conditioning of $\varphi$},
if

\medskip

(C1) for every $S\subseteq V$ and $z\in J^S$, $\varphi_z$ is a measure on
$(\Omega^{V\setminus S},\AA^{V\setminus S})$;

\medskip

(C2) for every $S\subseteq V$ and $B\in\AA^{V\setminus S}$, the value
$\varphi_z(B)$ is a measurable function of $z\in \Omega^S$;

\medskip

(C3) for every $T\subseteq S\subseteq V$, for every $z\in \Omega^T$, $B\in
\AA^{S\setminus T}$ and $C\in\AA^{V\setminus S}$,
\[
\varphi_z(B\times C)=\intl_{B} \varphi_{zy}(C) \,d\varphi_z^{S\setminus T}(y).
\]

As extreme cases, $\varphi_\emptyset=\varphi$, and $\varphi_z=\delta_z$ (the
Dirac measure) for $z\in \Omega^V$.

For a fixed set $S\subseteq V$, the family $\{\varphi_z:~z\in \Omega^S\}$ is a
disintegration of the measure $\varphi$ according to the marginal $\varphi^S$.
The conditioning as defined above means a bit more: first, it is well-defined
for all $z\in \Omega^S$, not just almost everywhere; second, it is defined
simultaneously for all marginals $\varphi^S$, with compatibility condition
(C3).

If $V$ is the set of nodes of a graph $G$, we can define an important
probabilistic property of conditionings. A conditioning $(\varphi_z)$ is {\it
Markovian}, if for every $S\subseteq V$ and $z\in \Omega^S$, the measure
$\varphi_z$ is multiplicative over the connected components of $G\setminus S$.
If, in particular, $\varphi$ is a probability distribution, and $G\setminus S$
has connected components $G_1,\dots,G_r$, then $\varphi_z|_{G_1},\dots,
\varphi_z|_{G_1}$ are independent.

\subsection{Graphons}\label{SSEC:GRAPHONS}

We conclude this section with a brief survey of related constructions for
graphons, partly as analogues for the orthogonality graph (which is not a
graphon), but also for later reference. A {\it graphon} is a symmetric
integrable function $W:~\Omega^2\to\R_+$, where $(\Omega,\AA,\pi)$ is a
standard Borel probability space. In the theory of dense graph limits, graphons
are bounded by $1$, but since then much of the theory has been extended to the
unbounded case.

Given a graphon $W$ and a finite simple graph $G=(V,E)$, we define a function
$W^G:~\Omega^V\to\R_+$ for $x=(x_i:~i\in V)$ by
\begin{equation}\label{EQ:WG-FUNC}
W^G(x) = \prod_{ij\in E} W(x_i,x_j).
\end{equation}
Sometimes it will be convenient to use this notation for a single edge $e=ij\in
E$: $W^e(x)=W(x_i,x_j)$. The function $W^G$ defines a measure $\eta_W^G$ as its
density function:
\[
\eta_W^G(A) = \intl_A W^G(x)\,d\pi^V(x)\qquad(A\in\AA^V),
\]
and the {\it subgraph density}
\begin{equation}\label{EQ:WG-DENS}
t(G,W) = \eta_W^G(\Omega^V) = \intl_{\Omega^V} W^G(x)\,d\pi^V(x).
\end{equation}
For bounded graphons this is always finite, but in general, it may be infinite.

We call a graphon {\it $1$-regular}, if $\int_\Omega W(x,y)\,d\pi(y) =1$ for
every $x$. For a $1$-regular graphon, the function $W(x,.)$ can be considered
as the density function of a probability distribution $\nu_x$ on
$(\Omega,\AA)$, which defines a step from $x\in \Omega$ of a time-reversible
Markov chain. Let us make $n$ independent steps, each from the same point $x$,
to points $x_1,\dots,x_n$. The joint distribution of $(x,x_1,\dots,x_n)$ has
density function $W(x,x_1)\dots W(x,x_n)$, and if $x$ is chosen randomly from
$\pi$, then the analogue of Lemma \ref{LEM:N-LOOSE} says that the joint
distribution of $(x_1,\dots,x_n)$ has density function
\begin{equation}\label{EQ:KSTEP-GRAPHON}
s(x_1,\dots,x_n)=\intl_\Omega W(x,x_1)\dots W(x,x_n)\,d\pi(x).
\end{equation}

A Markovian conditioning of $\eta_W^G$ can be constructed as the family of
measures $\{\eta^z:~z\in \Omega^S,\,S\subseteq V\}$, with density functions
\[
t_z(G,W) = \intl_{\Omega^{V\setminus S}} W^G(y,z)\,d\pi^{V\setminus S}(y).
\]

\section{The main construction}

\subsection{Swapping lemmas}

For the next (main) lemma, we need some geometric preparation. We fix the
dimension $d$. Let $L_i$ ($i=1,2$) be linear subspaces of $\R^d$ of dimension
$d_i\ge 2$. Let $x_i\in L_i\cap S^{d-1}$, and let $\wh{x}_1$ and $\wh{x}_2$ be
the orthogonal projections of $x_1$ onto $L_2$ and of $x_2$ onto $L_1$,
respectively. Define
\[
\Omega=\{(x_1,x_2)\in L_1\times L_2:~x_1\perp x_2,~,\wh{x}_1,\wh{x}_2\not=0\}.
\]

\begin{lemma}\label{LEM:SWAP}
Let $X_i$ be a random vector from $L_i\cap S^{d-1}$, and let $X_1'$ and $X_2'$
be random vectors from the spheres $L_1\cap S^{d-1}\cap X_2^\perp$ and $L_2\cap
S^{d-1}\cap X_1^\perp$, respectively. Let $\rho_1$ and $\rho_2$ be the
distributions of $(X_1,X_2')$ and $(X_1',X_2)$, respectively. Then $\rho_1$ and
$\rho_2$ are mutually absolutely continuous on $\Omega$, and
\[
\frac{d\rho_2}{d\rho_1}(x_1,x_2) =\frac{A_{d_1-1}A_{d_2-2}}{A_{d_2-1}A_{d_1-2}}\,\frac{|\wh{x}_2|}{|\wh{x}_1|}
= \frac{A_{d_1-2}}{A_{d_2-2}}\,\frac{|\wh{x}_2|}{|\wh{x}_1|}.
\]
\end{lemma}

\begin{proof}
The first assertion follows by the considerations in \cite{LSS,GeomBook}, and
also from the computations below.

For a nonzero vector $u\in\R^d$, we set $u^0=u/|u|$. Let $(x_1,x_2)\in\Omega$,
let $u_1=x_1, u_2=\wh{x}_2^0, u_3,\dots,u_{d_1}$ be an orthonormal basis in
$L_1$, and select an orthonormal basis $v_1,\dots, v_{d_2}$ in $L_2$
analogously. Let $\|.\|_\infty$ denote the $\ell_\infty$ norm on each of $L_1$
and $L_2$ in these bases. Let $T_i$ be the tangent space of the unit sphere
$U_i$ of $L_i$ at $x_i$ (as an affine subspace of $L_i$ containing $x_i$). Fix
an $\eps>0$. Let $B_i$ be the cube $x\in T_i:~\|x-x_i\|_\infty\le \eps$, and
let $B_i'$ denote the projection of $B_i$ onto the sphere $U_i$ from the
origin.

For $y_1\in B_1$, consider the linear subspace $H=H(y_1)= \{y_2\in L_2:~y_1\T
y_2=0\}$ and the affine subspace $H'=H'(y_1) = \{y_2\in L_2:~x_1\T y_2 + x_2\T
y_1=0\}$. Note that the equation defining $H'(y_1)$ can be written as
$H'(y_1)=\{y\in L_2:~\wh x_1\T y + \wh x_2\T y_1=0\}$, since $x_1-\wh x_1\perp
y$ and $x_2-\wh x_2\perp y_1$. Furthermore, $x_2-\wh x_2\perp x_1$ by the
orthogonality of the projection, so $\wh{x}_2 = x_2 - (x_2-\wh x_2) \perp x_1$.
We claim that these two subspaces are almost the same:

\begin{claim}\label{CLAIM:HH}
There is a constant $C>0$ independent of $\eps$ such that $d(y_2,H')<C\eps^2$
for every $y_2\in H\cap B_2$, and $d(y_1,H')<C\eps^2$ for every $y_1\in H\cap
B_2$.
\end{claim}

We use the identity
\begin{equation*}\label{EQ:YYEE}
y_1\T y_2 - (x_1\T y_2 + x_2\T y_1) = (y_1-x_1)\T(y_2-x_2)
\end{equation*}
(all asymptotic statements concern $\eps\to0$). Here $\|y_i-x_i\| = O(\eps)$,
so the right hand side is $O(\eps^2)$. Up to sign, the first term on the left
is $\|y_1\|\,d(y_2,H)$, while the second term is $\|\wh x_1\|d(y_2,H')$. If
either one of these is $0$, the other one is $O(\eps^2)$.

Let $X_i$ and $X_i'$ be generated as in the statement of the Lemma. Then
\[
\Pr\big((X_1,X'_2)\in B'_1\times B'_2\big) = \Pr(X'_2\in B'_2~|~ X_1\in B'_1) \Pr(X_1\in B'_1).
\]
Here
\[
\Pr(X_1\in B'_1) = \frac{\lambda_{d_1-1}(B_1')}{A_{d_1-1}} \sim \frac{\lambda_{d_1-1}(B_1)}{A_{d_1-1}}
= \frac{(2\eps)^{d_1-1}}{A_{d_1-1}}.
\]
(where $\lambda_k$ denotes the $k$-dimensional volume in $\R^d$). The first
factor is more complicated. For a fixed $y_1\in B_1$, we have
\[
\Pr(X'_2\in B'_2~|~ X_1=y^0_1) = \frac{\lambda_{d_2-2}(B_2'\cap H(y_1))}{A_{d_2-2}}
\sim \frac{\lambda_{d_2-2}(B_2\cap H(y_1))}{A_{d_2-2}}.
\]
We want to compare $B_2\cap H(y_1)$ and $B_2\cap H'(y_1)$. The hyperplane
$H'(y_1)$ in $L_2$ is orthogonal to the edge $v_2$ of the cube $B_2$, and hence
it either avoids $B_2$ or intersects it in a set isometric with a facet. Using
the fact that $H(y_1)$ and $H'(y_1)$ are very close, we get that there is a
$C>0$, independent of $\eps$, such that
\begin{equation}\label{EQ:HH1}
\text{if $d(x_2,H'(y_1))<\eps-C\eps^2$, then $\lambda_{d_2-2}(B_2\cap H(y_1))\sim
(2\eps)^{d_2-2}$,}
\end{equation}
and
\begin{equation}\label{EQ:HH2}
\text{if $d(x_2,H'(y_1))>\eps+C\eps^2$, then $\lambda_{d_2-2}(B_2\cap H(y_1))=0$.}
\end{equation}
In the modified equation defining $H'$, the coefficient vector $\wh x_1\in
L_2$, and hence
\[
d(x_2,H'(y_1)) = \frac1{|\wh x_1|} \Big|\wh x_1\T
x_2 + \wh x_2\T y_1\Big| =  \frac{|\wh x_2|}{|\wh x_1|}|u_1\T y_1|.
\]
Hence
\[
\Pr(X_2'\in B'_2~|~X_1=y_1^0) \sim
  \begin{cases}
    \displaystyle\frac{(2\eps)^{d_2-2}}{A_{d_2-2}}, & \text{if $|u_1\T y_1| < (\eps-C\eps^2)|\wh{x}_1|/ |\wh{x}_2|$}, \\
    0, & \text{if $|u_1\T y_1| > (\eps+C\eps^2)|\wh{x}_1|/ |\wh{x}_2|$},\\
    O(\eps^{d_2-2}), & \text{otherwise}.
  \end{cases}
\]
The first option applies for a fraction of $\min\{1, (1-C\eps)
|\wh{x}_2|/|\wh{x}_1|\}$ of points of $B_1$. The third possibility occurs for a
negligible fraction of the points of $B_1$. Since the distribution of $X_1$ in
$B_1$ is almost uniform, we get
\[
\Pr\big((X_1,X'_2)\in B_1\times B_2\big)\sim \frac{(2\eps)^{d_1-1}}{A_{d_1-1}}  \frac{(2\eps)^{d_2-2}}{A_{d_2-2}}
\frac{\min(|\wh{x}_1|,|\wh{x}_2|)}{|\wh{x}_2|}.
\]
Similarly,
\[
\Pr\big((X'_1,X_2)\in B_1\times B_2\big)\sim \frac{(2\eps)^{d_2-1}}{A_{d_2-1}}  \frac{(2\eps)^{d_1-2}}{A_{d_1-2}}
\frac{\min(|\wh{x}_1|,|\wh{x}_2|)}{|\wh{x}_1|}.
\]
and so
\[
\frac{d\rho_2}{d\rho_1}(x_1,x_2) \sim \frac{\Pr\big((X_1',X_2)\in B_1\times B_2\big)}{\Pr\big((X_1,X'_2)\in B_1\times B_2\big)}
\sim\frac{A_{d_1-1}A_{d_2-2}}{A_{d_2-1}A_{d_1-2}} \frac{|\wh{x}_2|}{|\wh{x}_1|}.
\]
Letting $\eps\to0$, the lemma follows.
\end{proof}

Let $p:~V\to [n]$ be a bijection, defining an ordering of the nodes of a graph
$G=(V,E)$, let $N_p(u)=\{w:~p(w)<p(u),\ uw\in E\}$, and let $d_p(u) =
|N_p(u)|$. To simplify notation, for a map $x:~V\to\R^d$, we write
$x_p(u)=x|_{N_p(u)}$.

We recall more formally the construction of an ortho-homomorphism from the
Introduction. Let $v\in V$, $S=\{u\in V:~p(u)<p(v)\}$, and suppose that we
already have an ortho-homomorphism $(x_u:~u\in S)$ in general position for the
subgraph $G[S]$. The vectors in $S^{d-1}$ orthogonal to every $x_i$ with $i\in
N_p(v)$ form a sphere of dimension at least $(d-1)-d_p(v)\ge 0$; we choose a
vector $x_v$ on this sphere randomly. Repeating this until $x_v$ is defined for
every $v\in V$, we get an ortho-homomorphism, which we call a {\it random
sequential ortho-homomorphism} of $G$. Let $\rho_p$ be the distribution of this
ortho-homomorphism. By \cite{LSS}, this ortho-homomorphism is in general
position almost surely. The main step in the proof was that flipping two
consecutive nodes in the ordering, we may get a possibly different distribution
on ortho-homomorphisms, but this new new distribution is absolutely continuous
with respect to the previous one. In the next lemma, we give an explicit
formula showing this.

\begin{lemma}\label{LEM:SWAP2}
Let $r$ be obtained from the ordering $p$ by flipping two consecutive adjacent
nodes $u$ and $v$, where $p(v)=p(u)+1$. Then
\[
\frac{d\rho_r}{d\rho_p}(x) =
\frac{A_{d-d_p(u)-1}A_{d-d_p(v)-1}}{A_{d-d_r(u)-1}A_{d-d_r(v)-1}}\,
\frac{\Det(x_r(u))\Det(x_r(v))}{\Det(x_p(u))\Det(x_p(v))}.
\]
\end{lemma}

\begin{proof}
We apply Lemma \ref{LEM:SWAP} with $L_1=N_p(u)^\perp$ and $L_2=N_r(v)^\perp$.
Then $\dim(L_1)=d-d_p(u)\ge d-\deg(u)+1\ge 2$ (since $v$ is not counted in
$d_p(u)$), and similarly $\dim(L_2)=d-d_r(v)\ge 2$. Also note that
$d_r(u)=d_p(u)+1$, $d_r(v)=d_p(v)-1$, and $d_r(w)=d_p(w)$ for every
$w\not=u,v$. Since $x_p(u)$ is a basis in $L_1^\perp$ and
$x_r(v)=x_p(v)\setminus\{u\}$ is a basis in $L_2^\perp$, the length of the
orthogonal projection of $x_u$ onto $L_2$ is
\[
\frac{\Det(x_r(v)\cup\{u\})}{\Det(x_r(v))} = \frac{\Det(x_p(v))}{\Det(x_r(v))}
\]
and the length of orthogonal projection of $x_v$ onto $L_1$ is
\[
\frac{\Det(x_p(u)\cup\{v\})}{\Det(x_p(u))} = \frac{\Det(x_r(u))}{\Det(x_p(u))}.
\]
This implies, in particular, that these projections are nonzero, and so we can
apply Lemma \ref{LEM:SWAP}. Since the order in which $x_u$ and $x_v$ are chosen
does not influence the distribution of $(x_w:~p(w)<p(u))$ and the distribution
of $(x_w:~p(w)>p(v))$ conditional on $(x_w:~p(w)\le p(v))$, we get that
\begin{align*}
\frac{d\rho_p}{d\rho_r}(x) &=
\frac{A_{d-d_p(u)-1}A_{d-d_p(v)-1}}{A_{d-d_r(v)-1}A_{d-d_r(u)-1}}
\left(\frac{\Det(x_r(u))}{\Det(x_p(u))}\left/\frac{\Det(x_p(v))}{\Det(x_r(v))}
\right.\right),
\end{align*}
proving the lemma.
\end{proof}

\subsection{Order-independent measure}\label{SSEC:ORD-INDEP}

Let $p:~V\to [n]$ be an ordering of the nodes of a graph $G$, and let
$x:~V\to\R^d$ be an orthogonal representation in general position. Using the
functions defined in \eqref{EQ:SNSD}, let
\begin{align}\label{EQ:F-DEF}
f_p(x) &=  \prod_{v\in V} s_{d,n}(x_p(v))
= \prod_{v\in V} \frac{A_{d-1}^{d_p(v)-1}A_{d-d_p(v)-1}}{A_{d-2}^{d_p(v)}}\,\frac{1}{\Det(x_p(v))}  \nonumber\\
&= \frac{A_{d-1}^{|E|-|V|}}{A_{d-2}^{|E|}} \prod_{v\in V} \frac{A_{d-d_p(v)-1}}{\Det(x_p(v))}.
\end{align}
We define a measure $\varphi_p= f_p\cdot\rho_p$ on $\Sigma$; more explicitly,
\begin{equation}\label{EQ:PHI-DEF}
\varphi_p(A) = \intl_A f_p\,d\rho_p.
\end{equation}
The following lemma is the main property of this construction.

\begin{lemma}\label{LEM:MAIN}
The measure $\varphi_p$ is independent of the ordering $p$.
\end{lemma}

By this lemma, we can denote $\varphi_p$ simply by $\varphi$ or $\varphi_G$. We
can think of $\varphi$ either as a measure on $\Sigma_G$, or as a measure on
$(\R^d)^V$ concentrated on $\Sigma_G$.

\begin{proof}
It suffices to check that if $r$ is the permutation obtained from $p$ by
swapping two consecutive nodes $u$ and $v$, then $\varphi_p=\varphi_r$. If $u$
and $v$ are nonadjacent, then this is trivial: $\rho_p=\rho_r$ and
$N_p(w)=N_r(w)$ for every node $w$, and hence $f_p=f_r$. So suppose that $uv\in
E$, and (say) $p(v)=p(u)+1$. Then
\[
\frac{d\varphi_p}{d\varphi_r}(x) = \frac{f_p(x)}{f_r(x)}\cdot\frac{d\rho_p}{d\rho_r}(x).
\]
Here
\[
\frac{f_p(x)}{f_r(x)}
= \frac{A_{d-d_p(u)-1}A_{d-d_p(v)-1}}{A_{d-d_r(v)-1}A_{d-d_r(u)-1}}\,
\frac{\Det(x_r(u))\Det(x_r(v))}{\Det(x_p(u))\Det(x_p(v))}
\]
by definition \eqref{EQ:F-DEF}. Substituting for the second factor from Lemma
\ref{LEM:SWAP2}, we get
\[
\frac{d\varphi_p}{d\varphi_r}(x) =1.
\]
Since this holds for all $x\in\Sigma$, this proves that $\varphi_r=\varphi_p$.
\end{proof}

The measure $\varphi$ is not always finite: in Section \ref{SEC:BIPARTITE} we
show that it is finite for every even cycle longer than $4$ in dimension $3$,
but infinite for the $3$-cube in dimension $4$. The measure is, however, finite
on compact subsets of $\Sigma_G$: the denominator in \eqref{EQ:F-DEF} remains
bounded away from zero. It is easy to see that $\varphi$ is a Radon measure.

We will also be interested in the {\it ortho-homomorphism number} (of graph $G$
in dimension $d$)
\begin{equation}\label{EQ:TGD-DEF}
t(G,d)=\varphi(\Sigma_G)=\frac{A_{d-1}^{|E|-|V|}}{A_{d-2}^{|E|}}
\intl_{(S^{d-1})^V}\prod_{v\in V} \frac{A_{d-d_p(v)-1}}{\Det(x_p(v))}\,d\pi^V.
\end{equation}
Let us note that $t(G,d)$ is positive for every $d$-sparse graph $G$; but it
may be infinite. If $t(G,d)$ is finite, then we can scale $\varphi$ to get a
probability measure on ortho-homomorphisms of $G$ into $S^{d-1}$.

The measure $\varphi$ has a natural conditioning. We can think of the
construction of the measure $\varphi$ as follows: Choose the vectors $x_i$ in
any given order according to the random sequential rule; whenever $x_i$ is
chosen, we multiply the density function by $s_{d,n}(x_p(v))$ (which is
determined by the previous nodes). An important consequence of this fact is
that if we stop when a subset $S$ of nodes has been processed, the vectors
selected and the density function computed up to this point define an
ortho-homomorphism from the measure $\varphi_{G[S]}$.

For $z\in \Sigma_{G[S]}$, we construct a measure $\varphi_z$ on
$\Sigma_{G[V\setminus S]}$ by continuing the random sequential choice.
Formally, let $p$ be any ordering of the nodes of $G$ starting with $S$; extend
$z$ to an ortho-homomorphism $x$ of $G$ in $\R^d$ by random sequential choice;
let $\rho_{z,p}$ be the distribution of this extension. Define the density
function the measure $\varphi_z$ on $\Sigma_{G[V\setminus S]}$ by
\begin{equation}\label{EQ:PHI-DEFP}
\varphi_z(A) = \intl_A \prod_{u\in V\setminus S}s_{d,n}(x_p(u))\,d\rho_{z,p}.
\end{equation}

\begin{lemma}\label{LEM:MARKOV}
The family $(\varphi_z:~z\in\Sigma_{G[S]},\,S\subseteq V)$ is a Markovian
conditioning of $\varphi$.
\end{lemma}

\begin{proof}
The fact that the family $(\varphi_z)$ is a conditioning follows from the
construction of $\varphi$ as described above.

The Markov property is easy as well. Let $S\subseteq V$, and let
$G_1,\dots,G_r$ be the connected components of $G\setminus S$. Let
$z\in\Sigma_{G[S]}$. Constructing the random extension of $z$ sequentially, we
see $\varphi_z|_{G_i}$ is independent of the vectors and density function
values of the other components $G_j$.
\end{proof}

Lemmas \ref{LEM:MAIN} and \ref{LEM:MARKOV} imply Theorem
\ref{THM:ORTHO-MARKOV1}.

The following related fact was observed and used (implicitly) in \cite{LSS}.

\begin{prop}\label{PROP:MARG-SUB}
Let $S\subseteq V(G)$, let $G'=G[S]$, and let $p$ be an ordering of the nodes
of $G$ starting with $S$. Then $\rho_p^S=\rho_{G',p}$, and $\varphi_G^S$ is
absolutely continuous with respect to $\varphi_{G'}$, and vice versa.
\end{prop}

\begin{proof}
The first assertion is obvious from the sequential construction of $\rho_p$. By
construction, $\varphi_{G'}$ and $\rho_{G',p}$ are mutually absolutely
continuous, and so are $\varphi_G$ and $\rho_p$. The second assertion implies
that their marginals $\rho_p^S=\rho_{G',p}$ and $\varphi_G^S$ are mutually
absolutely continuous, and hence so are $\varphi_G^S$ and $\varphi_{G'}$.
\end{proof}

\subsection{Explicitly computable examples}

\begin{example}[Trees]\label{EXA:TREE}
A simple example is a tree $F$. Let $p$ be a search order from a root $u$. Then
$d_p(v)=1$ for every $v\not= u$, and $\Det(x_p(u))=1$ for every $u$. Hence
$f_p(x)\equiv 1$, $\rho_p$ is the same distribution for every search order, and
$\varphi=\rho_p$. Thus the measure $\varphi(F,d)$ is a well defined probability
distribution, and $t(G,d)=1$. The more general example of bipartite graphs will
be discussed in Section \ref{SEC:BIPARTITE}.
\end{example}

\begin{example}[Triangles]\label{EXA:TRIANGLE}
Let $d=3$ and $G=K_3$, with the nodes labeled $1,2,3$. Then $(x_1,x_2)$ is
uniformly distributed on orthogonal pairs. It follows that $\Det(x_1,x_2)=1$,
and so all the Det's in the denominator of \eqref{EQ:TGD-DEF} are $1$. Hence
\begin{equation}\label{EQ:K3D}
t(K_3,d) = \frac{A_{d-1}A_{d-3}}{A_{d-2}^2} = \frac{(\pi/2)^{(-1)^d}\big((d-3)!!\big)^2}{(d-2)!!(d-4)!!}.
\end{equation}
In particular,  $t(K_3,3) = 2/\pi$ and $t(K_3,4)=\pi/4$. Other cycles will be
discussed in Sections \ref{SEC:BIPARTITE} and \ref{SSEC:CYCLE-SPEC}.
\end{example}

\begin{example}[Rigid Circuit Graphs]\label{EXA:RIGID}
We can get rid of the integration for all {\it rigid circuit graphs}, which
contain no induced cycles other than triangles. A well-known characterization
of these graphs is that their nodes can be ordered so that the neighbors of any
node $v$ preceding it spans a complete subgraph. Using this ordering $p$ to
compute $t(G,d)$ (where $d$ is large enough so that $G$ is $d$-sparse), we see
that the vectors in every $x(N_p(v))$ are mutually orthogonal, and so the Det's
in the denominator are $1$. Hence we get
\begin{equation}\label{EQ:RIGID-T}
t(G,d)=\frac{A_{d-1}^{|E|-|V|}}{A_{d-2}^{|E|}}\prod_{v\in V}  A_{d-d_p(v)-1}.
\end{equation}
In particular, we get a formula for complete graphs:
\begin{equation}\label{EQ:TGD-KR}
t(K_r,d)=\frac{A_{d-1}^{r(r-3)/2}}{A_{d-2}^{r(r-1)/2}} \prod_{i=1}^{r}A_{d-i}.
\end{equation}

We could use \eqref{EQ:SK-DEF} to express \eqref{EQ:RIGID-T} as $a\pi^b$, where
$a$ is rational and $b$ is an integer. Let $q$ denote the number of odd
``backward'' degrees $d_p(v)$. Then straightforward computation gives that
\[
b=\begin{cases}
   \displaystyle \frac{|E|-q}{2}, & \text{if $d$ is even}, \\[6pt]
   \displaystyle \frac{q-|E|}{2}, & \text{if $d$ is odd}.
  \end{cases}
\]
Surprisingly, this exponent does not depend on $d$ except for its sign. The
combinatorial significance of the rational coefficient $a$ would be interesting
to determine.
\end{example}

\begin{example}[Complete bipartite graphs]\label{EXA:C4}
Let $G=K_{a,b}$, where $a+b\le d$. Then, using \eqref{EQ:TGD-DEF} and Lemma
\ref{LEM:XEXPECT},
\begin{align}\label{EQ:BICOMPLETE}
t(G,d)&=\frac{A_{d-1}^{ab-b}A_{d-a-1}^b}{A_{d-2}^{ab}}
\intl_{(S^{d-1})^a}\frac{1}{\Det(x)^b}\,d\pi^a(x)\nonumber\\
&= \frac{A_{d-1}^{ab-b-a}A_{d-b-1}^aA_{d-a-1}^b A_{d-1}\cdots A_{d-a}}
{A_{d-2}^{ab} A_{d-b-1}\cdots A_{d-b-a}}.
\end{align}
This implies that $t(G,d)$ can again be expressed as a rational multiple of an
integer power of $\pi$, where the exponent of $\pi$ depends on the parity of
$d$ only.
\end{example}

\subsection{Bipartite graphs}\label{SEC:BIPARTITE}

Let $G$ be a bipartite graph with bipartition $V=U\cup W$. The construction of
the measure $\varphi$ can be carried out by ordering the nodes starting with
$U$, to get the reference ordering $p$ of $V$. If $x$ is a random point from
$\rho_p$, then $x_u$ ($u\in U$) are independent random vectors in $S^d$, and
$f_p$ depends only on these variables $x_u$. Furthermore,
$s_{d,n}(x_p(u))=A_{d-1}$ for $u\in U$. Hence (canceling $A_{d-1}^{|U|}$)
\[
f_p(x) =  \frac{A_{d-1}^{|E|-|W|}}{A_{d-2}^{|E|}}
\prod_{v\in W} \frac{A_{d-d(v)-1}}{\Det(x(N(v)))},
\]
and
\begin{align}\label{EQ:BIPART}
t(G,d)&=\intl_{(S^{d-1})^V} f_p\,d\rho_p = \intl_{(S^{d-1})^U} \prod_{v\in W} s_{d,d_p(v)}(x)\,d\pi^U\nonumber\\
&= \frac{A_{d-1}^{|E|-|W|}}{A_{d-2}^{|E|}}
\intl_{(S^{d-1})^U} \prod_{v\in W}\frac{A_{d-d(v)-1}}{\Det(x(N(v)))}\,d\pi^U
\end{align}
We'll see examples where this number is finite and also where this number is
infinite.

\begin{remark}\label{REM:HYPERPLANES}
An ortho-homomorphism of a bipartite graph has the following geometric
interpretation. Consider the points of $U$ as vectors in $\R^d$ (as before),
but the points of $W$ as normal vectors of hyperplanes. Orthogonality
translates to incidence. For example, a representation of $C_6$ in $\R^3$ is a
simplicial cone, with three rays (the vectors in their direction having unit
length), and the normals of the three faces (again, of unit length).
\end{remark}

\begin{remark}\label{REM:SIDOR}
We'll see (Corollary \ref{appthmcor}) that Sidorenko's conjecture would imply
the inequality
\begin{equation}\label{EQ:SIDOR}
t(G,d)\ge 1
\end{equation}
for every $d$-sparse bipartite graph $G$. It would be interesting to prove this
inequality at least in this special case.
\end{remark}

As a simple but important special class of bipartite graphs, the subdivision
$G=H'$ of a simple graph $H$ by one node on each edge is a bipartite graph.
Then
\begin{equation}\label{EQ:LINE0}
t(H',d)=\frac{A_{d-1}^{|E|} A_{d-3}^{|E|}}{A_{d-2}^{2|E|}}\intl_{(S^{d-1})^U}
\prod_{ij\in E}\frac{1}{\Det(x_i,x_j)}\,d\pi^V.
\end{equation}

We can use this special case to justify considering $t(G,d)$ as the density of
$G$ in the orthogonality graph. The function
\begin{equation}\label{EQ:2-STEP}
s_{d,2}(x,y) = \frac{A_{d-1} A_{d-3}}{A_{d-2}^2} \frac1{\Det(x,y)}
\end{equation}
defines a graphon $(S^{d-1},\pi,s_{d,2})$. Let $n=|V(H)|$, $m=|E(H)|$, and
assume that all degrees of $H$ are bounded by $d-1$. Then the
ortho-homomorphism number can be expressed as follows.
\begin{equation}\label{EQ:THD}
t(H',d) =  \intl_{(S^{d-1})^V}\prod_{ij\in E(H)} s_{d,2}(x(N(v)))\,d\pi^V(x) =t(H,s_{d,2}).
\end{equation}
Since $t(H',W)=t(H,W\circ W)$ for any graphon $W$, this justifies to consider
$t(G,d)$ as the homomorphism density of $G$ in the orthogonality graph (with
the edge measure scaled to a probability measure).

\begin{example}\label{EXA:C6}
Let $d=3$ and $G=C_6=K_3'$. Then the conditions above are satisfied, and
\eqref{EQ:BIPART} gives that
\[
t(C_6,3) = \frac{8}{\pi^3} \intl_{(S^2)^3}\frac1{\Det(x_1,x_2)\Det(x_2,x_3)\Det(x_3,x_1))}\,d\pi^3(x).
\]
Let $\measuredangle(x_1,x_3)=\alpha$, $\measuredangle(x_2,x_3)=\beta$, and
$\measuredangle(x_1,x_2)=\gamma$. Let $\theta$ denote the (unsigned) angle
between the planes $\lin(x_1,x_3)$ and $\lin(x_2,x_3)$. By the spherical cosine
theorem, $\gamma=\gamma(\alpha,\beta,\theta)$ is given by
\begin{equation}\label{EQ:COSINE}
\cos\gamma = \cos\alpha\cos\beta - \sin\alpha\sin\beta\cos\theta.
\end{equation}
Fixing $x_3$, it is easy to see that the angles $\alpha$, $\beta$ and $\theta$
are independent random variables with values in $[0,\pi]$; their density
functions are $\frac12\sin\alpha$, $\frac12\sin\beta$ and $1/\pi$,
respectively. Hence
\begin{align*}
t(C_6,3) = \frac{8}{\pi^3} \intl_{[0,\pi]^3} \frac{\frac12\sin\alpha\,\frac12 \sin\beta\,\frac1{\pi}}
{\sin\alpha\,\sin\beta\,\sin\gamma}\,d\alpha\,d\beta\,d\theta
= \frac{2}{\pi^4} \intl_{[0,\pi]^3} \frac{1}{\sin\gamma}\,d\alpha\,d\beta\,d\theta.
\end{align*}
Substituting from \eqref{EQ:COSINE},
\begin{equation}\label{EQ:TC6}
t(C_6,3) = \frac{2}{\pi^4} \intl_{[0,\pi]^3} \frac{d\alpha\,d\beta\,d\theta}{\sqrt{1-\big(\cos\alpha\cos\beta
- \sin\alpha\sin\beta\cos\theta\big)^2}}.
\end{equation}
\end{example}

\section{Finiteness}

The value $t(G,d)$, as defined by the integral in \eqref{EQ:TGD-DEF}, may be
finite or infinite even for $d$-sparse graphs, as we will show below. In this
section, we study the issue of finiteness. We only address this issue for
bipartite graphs. Further exact formulas, based on spectral methods, will be
given in the next section.

\subsection{A general bound}

A general upper bound on $t(G,d)$ can be obtained by applying the following
generalized H\"older inequality \cite{Fin}.

\begin{lemma}\label{LEM:GHOLDER}
Let $f_1,\dots,f_m:~\Omega^n\to\R$ be measurable $n$-variable functions on some
probability space $(\Omega,\AA,\pi)$, such that $f_i$ depends only on a subset
$B_i\subseteq[n]$ of the variables. Let $p_1,\dots,p_m\ge1$ such that
\[
\sum_{i:\,B_i\ni j}\frac1{p_i} \le 1\qquad(j=1,\dots, n).
\]
Then
\[
\intl_{\Omega^n} f_1\dots f_m\,d\pi^n \le \|f_1\|_{p_1}\dots\|f_m\|_{p_m}.
\]
\end{lemma}

Of course, the lemma is only interesting if the right hand side is finite,
i.e., if $f_i\in L_{p_i}(\Omega^m)$ for every $i$. Applying it to the
expression \eqref{EQ:BIPART}, we get
\begin{equation}\label{EQ:TGDS}
t(G,d)\le \prod_{v\in W} \|s_{d,\deg(v)}\|_{p_v},
\end{equation}
where the numbers $p_v$ must satisfy
\begin{equation}\label{EQ:VNUP}
\sum_{v\in N(u)}\frac1{p_v}\le 1
\end{equation}
for all $(u\in U)$. The bound is finite when $s_{d,\deg(v)}\in L_{p_v}$ for all
$v\in W$; as noted in Section \ref{SEC:DETS}, this happens if and only if
\begin{equation}\label{EQ:PVDD}
p_v\le d-\deg(v)\qquad(v\in W).
\end{equation}
The upper bound in \eqref{EQ:TGDS} can be expressed explicitly using Lemma
\ref{LEM:XEXPECT}, but it is not really appealing. However, the finiteness
result is worth stating:

\begin{theorem}\label{THM:BIP-FINITE1}
Let $G$ be a bipartite graph with bipartition $(U,W)$, and suppose that
\begin{equation}\label{EQ:FIN-BOUND}
\sum_{v\in N(u)} \frac1{d-\deg(v)} \le 1
\end{equation}
for all $u\in U$. Then $t(G,d)$ is finite.
\end{theorem}

\begin{proof}
The condition implies that $G$ is $d$-sparse, and so $t(G,d)$ is well defined.
Choosing $p_v=d-\deg(v)$, \eqref{EQ:VNUP} and \eqref{EQ:PVDD} are satisfied.
\end{proof}

A special case when this condition is satisfied and that is easier to handle is
the following.

\begin{corollary}\label{COR:BIP-FINITE2}
Let $G$ be a bipartite graph with bipartition $(U,W)$, and suppose that all
degrees in $U$ are bounded by $a$ $(1\le a\le d-2)$, and all degrees is $W$ are
bounded by $d-a$. Then $t(G,d)$ is finite.
\end{corollary}

\subsection{Subdivisions}

In this section we show:

\begin{theorem}\label{THM:SUBDIV-FIN}
If $G$ is the subdivision (with one node on each edge) of a simple graph $H$
with maximum degree $d-1$, then $t(G,d)$ is finite.
\end{theorem}

A notable special case for $d=3$ is the cycle $C_{2k}$, as the subdivision of
$C_k$ $(k\ge 3)$. Note that this Proposition does not follow from Corollary
\ref{COR:BIP-FINITE2} (only if the degrees are strictly smaller than $d-1$).

We need a simple lemma in elementary graph theory.

\begin{lemma}\label{LEM:TREE}
Let $G=(V,E)$ be a simple connected graph on $n\ge 2$ nodes, with all degrees
at most $D$, let $e_1\in E$ and $w:~E\to\R_+$. Then there is a spanning tree
$F$ of $G$, and integers $(k_e:~e\in E(F)$ such that $k_{e_1}=1$, $k_e\le D$
for all $e\in E(F)$, $\sum_e k_e = |E(G)|$, and
\[
\sum_{e\in E} w(e) \le \sum_{e\in E(F)} k_e w(e).
\]
\end{lemma}

\begin{proof}
Let $F$ be a maximum weight spanning tree. It suffices to define a map
$\phi:~E(G)\to E(F)$ such $w(\phi(e))\ge w(e)$, $k_e=|\phi^{-1}(e)|\le D$, and
$|\phi^{-1}(e_1)|=1$. For any search order $(v_1,\dots,v_n)$ of $F$ starting
with $e_1=v_1v_2$, we map each edge $v_iv_j$ ($i<j$) to the (unique) edge
$v_jv_{j'}$ of $F$ with $j'<j$. It is easy to see that this map satisfies our
requirements: at most $D$ edges are mapped onto any edge of $F$, no edge other
than itself is mapped onto $e_1$, and if $v_iv_j$ is mapped onto $v_{j'}v_j$,
then $w(v_iv_j) \le w(v_{j'}v_j)$, because otherwise replacing the edge
$v_{j'}v_j$ by $v_iv_j$, we would get a tree with larger weight than $F$.
\end{proof}

\begin{proof*}{Theorem \ref{THM:SUBDIV-FIN}}
By identity \eqref{EQ:THD}, we have
\[
t(G,d) = t(H,W),
\]
where $W=s_{d,2}$ defines a graphon on $S^{d-1}$. Lemma \ref{LEM:TREE}, applied
to the logarithm of $W$, gives that for every $x\in (S^{d-1})^V$ there is a
spanning tree $F$ of $G$ and integers $(k_e:~e\in E(F))$ such that $k_e\le
d-1$, $k_{e_1}=1$, $\sum_e k_e = |E(G)|$, and
\[
W^G(x) \le \prod_{e\in E(F)} (W^e(x))^{k_e}.
\]
Since $W$ is bounded from below, this implies that
\begin{equation}\label{EQ:WGWF}
W^G(x) \le C_0 W^{e_1}(x) W^{F\setminus e}(x)^{d-1}
\end{equation}
for some constant $C_0$ independent of $x$. For a spanning tree $F$ of $G$ with
an ordered edge set $E(F)=\{e_1,e_2,\dots,e_{n-1}\}$, let $Y_{F}$ denote the
set of points $x\in (S^{d-1})^V$ for which $W^{e_1}(x)\ge W^{e_2}(x)\ge\dots$
and \eqref{EQ:WGWF} is satisfied. By the above, $\cup_F Y_F = (S^{d-1})^V$, and
so
\begin{align*}
t(G,W) &= \intl_{(S^{d-1})^V} W^G\,d\pi^V \le \sum_{F} \intl_{Y_{F}} W^G\,d\pi^V\\
&\lesssim \sum_{F} \intl_{Y_{F}}  W^{e_1}(x) W^{F\setminus e_1}(x)^{d-1}\,d\pi^V(x)\\
&\lesssim \max_{F} \intl_{Y_{F}}  W^{e_1}(x) W^{F\setminus e_1}(x)^{d-1}\,d\pi^V(x)
\end{align*}
So it suffices to prove that
\begin{equation}\label{EQ:TFFIN}
\intl_{Y_{F}}  W^{e_1}(x) W^{F\setminus e_1}(x)^{d-1}\,d\pi^V(x) =
\intl_{Y_{F}} W^{e_1}(x) W^{e_2}(x)^{d-1}\dots W^{e_{n-1}}(x)^{d-1} \,d\pi^V(x)
\end{equation}
is finite for every edge-ordered spanning tree $F$.

Disregarding the condition on the ordering of the edges, the random variables
$W^{e_i}(x)$ are independent. Indeed, selecting the images of the nodes in a
search order of the tree, each $W^{e_i}(x)$ will have the same distribution
even with one endpoint of $e_i$ already fixed, by symmetry. Let
$\vartheta_i(x)$ be the angle between $x_u$ and $x_v$, where $e_i=uv$. Then
\[
W^{e_i}(x) \lesssim \frac1{\sin\vartheta_i},
\]
and the density function of each $\vartheta_i(x)$ is
\[
f(\vartheta) \lesssim (\sin \vartheta)^{d-2}
\]
Let $T(F)$ denote the set of vectors $(\vartheta_1,\dots,\vartheta_{n-1})$ with
$0\le\vartheta_i\le\pi$ and $\sin\vartheta_1\le\dots\le\sin\vartheta_{n-1}$.
Then
\begin{align*}
\intl_{Y_{F}}& W^{e_1}(x) W^{e_2}(x)^{d-1}\dots W^{e_{n-1}}(x)^{d-1} \,d\pi^V(x)\\
&\lesssim \intl_{T(F)} \frac{(\sin \vartheta_1)^{d-2}\dots (\sin \vartheta_{n-1})^{d-2}}{\sin \vartheta_1
(\sin\vartheta_2)^{d-1}\dots(\sin\vartheta_{n-1})^{d-1}}\,d\vartheta_1\dots d\vartheta_{n-1}\\
&=  \intl_{T(F)} \frac{(\sin\vartheta_1)^{d-3}}{\sin\vartheta_2\dots\sin\vartheta_{n-1}}\,d\vartheta_1\dots d\vartheta_{n-1}\\
&\le \intl_{T(F)} \frac{1}{\sin\vartheta_2\dots\sin\vartheta_{n-1}}\,d\vartheta_1\dots d\vartheta_{n-1}
\end{align*}
Introducing the variables $\phi_i=\min(\vartheta_i,\pi-\vartheta_i)$ and the
set
\[
T'(F)=\{(\phi_1,\dots,\phi_{n-1}):~ 0\le\phi_i\le\pi/2,\ \phi_1\le\dots\le\phi_{n-1}\},
\]
we can go on as follows:
\begin{align*}
\intl_{T'(F)} &\frac{1}{\sin\phi_2\cdots\sin\phi_{n-1}}\,d\phi_1\dots d\phi_{n-1}
\lesssim\intl_{T'(F)} \frac{1}{\phi_2\cdots \phi_{n-1}}\,d\phi_1\dots d\phi_{n-1}\\
&=\intl_0^{\pi/2}\intl_0^{\phi_{n-2}} \dots \intl_0^{\phi_2} \frac{1}{\phi_2\cdots \phi_{n-1}}\,d\phi_1\dots d\phi_{n-1}\\
&=\intl_0^{\pi/2}\intl_0^{\phi_{n-2}} \dots \intl_0^{\phi_3} \frac{1}{\phi_3\cdots \phi_{n-1}}\,d\phi_2\dots d\phi_{n-1}\\
&=\dots=\intl_0^{\pi/2}1\,d\phi_{n-1} = \pi/2.
\end{align*}
This proves the theorem.
\end{proof*}

\begin{remark}\label{REM:BCCZ}
Theorem \ref{THM:SUBDIV-FIN} asserts that the subdivision of any simple graph
with all degrees at most $d-1$ has a finite Markovian probability distribution
on its ortho-homomorphisms. (This does not remain true for multigraphs, as
shown by the multigraph on two nodes connected by $d-1$ edges.)

As remarked after Lemma \ref{LEM:N-LOOSE}, $s_{d,2}\in L_p(S^{d-1},\pi)$ if and
only if $p<d-1$, so $W=s_{d,2}$ is an $L_p$-graphon for every $p< d-1$, as
defined by Borgs et al.~in \cite{BCCZ}. By one of the results of that paper,
all simple graphs with all degrees at most $d-2$ have a finite density in $W$;
our analysis shows that this remains valid for graphs with degrees bounded by
$d-1$ in the special case of $s_{d,2}$.
\end{remark}

\subsection{Paths and cycles}

Theorem \ref{THM:SUBDIV-FIN} implies that every even cycle $C_{2k}$ of length
at least $6$ has finite density in every dimension $d\ge 3$. We are going to
show that this holds for odd cycles without exception. But for later reference,
we start with discussing properties of ortho-homomorphisms of paths.

Let $P_k$ denote the path of length $k$. As we have seen (Example
\ref{EXA:TREE}), the ortho-homomorphism measure of paths is a probability
distribution. Let $\eta_d^k$ denote the marginal of this distribution on the
pair of endpoints ($k\ge1$). In particular, $\eta_d^1=\eta_d$. Easy properties
of $\eta_d^k$ are summarized in the next lemma.

\begin{lemma}\label{LEM:PATHS}
The distribution $\eta_d^k$ is absolutely continuous with respect to $\pi^2$
for all $d\ge 3$ and $k\ge 2$. The density function
$u_{d,k}(x,y)=(d\eta_d^k)/d\pi^2(x,y)$ is continuous for $k\le 5$ if $d=3$ and
for $k\ge 3$ if $d\ge 4$. For $k=2$, it has a singularity when $x\|y$; for
$d=3$ and $k=3$, it has a singularity when $x\perp y$; for $d=3$ and $k=4$, it
has a singularity when $x\|y$.
\end{lemma}

\begin{proof}
By Lemma \ref{LEM:N-LOOSE},
\[
u_{d,2}(x,y) = s_{d,2}(x,y) = \frac{A_{d-1} A_{d-3}}{A_{d-2}^2} \frac1{\sin(\measuredangle(x,y))},
\]
from which statements for $k=2$ are easily verified. It is easy to check that
\begin{equation}\label{EQ:U-REC}
u_{d,k+m}(x,y)=  \intl_{S^{d-1}} u_{d,k}(x,z)u_{d,m}(z,y)\,d\pi(z)
\end{equation}
for $k,m\ge 2$, and
\begin{equation}\label{EQ:U-REC1}
u_{d,k+1}(x,y)=  \intl_{S^{d-1}\cap x^\perp} s_{d,k}(z,y)\,d\pi_0(z),
\end{equation}
where $\pi_0$ is the uniform distribution on the $(d-2)$-dimensional sphere
$x^\perp\cap S^{d-1}$. Using this formula, we see that $u_{d,3}$ is a
continuous function for non-orthogonal pairs of points $(x,y)$, and it is not
hard to check that if $\eps=\measuredangle(x,y)-\pi/2 \to 0$, then
\[
u_{d,3}(x,y) =
  \begin{cases}
    O(\log\eps) & \text{if $d=3$}, \\
    O(1), & \text{if $d>3$}.
  \end{cases}
\]
From this it follows that $u_{d,k}$ is bounded (even continuous) for all $k\ge
3$ if $d\ge 4$. If $d=3$, then \eqref{EQ:U-REC} implies that $u_{d,4}$ still
has singularity if $x=y$; for $\eps=\measuredangle(x,y)\to 0$, we have
$u_{d,4}(x,y)=O(\log\eps)$. Using \eqref{EQ:U-REC} again, we see that $u_{d,k}$
is bounded and continuous on $S^{d-1}\times S^{d-1}$ for all $k\ge 5$.
\end{proof}

\begin{theorem}\label{COR:C-FIN}
The ortho-homomorphism density $t(C_k,d)$ is finite except if $d=3$ and $k=4$.
\end{theorem}

\begin{proof}
For even cycles longer than $4$ we already know this by Theorem
\ref{THM:SUBDIV-FIN}; also for $k=3$, by the computations of Example
\ref{EXA:TRIANGLE}. Let $k=2r+1$, $r\ge2$. Then, using any ordering of the
nodes, we get that
\[
t(C_{2r+1},d) = \intl_{S^{d-1}} u_{d,2}u_{d,2r-1} \,d\pi^2.
\]
Here $u_{d,2}<C_1$ on $S_1=\{(x,y):~\pi/4\le \measuredangle(x,y)\le 3\pi/4\}$
and $u_{d,2r-1}<C_2$ on $S_2=(S^{d-1})^2\setminus S_1$, by Lemma
\ref{LEM:PATHS}, for some constants $C_i$. Thus
\begin{align*}
t(C_{2r+1},d) &= \intl_{S^{d-1}} u_{d,2}u_{d,2r-1} \,d\pi^2 \le
C_1 \intl_{S_1} u_{d,2r-1} \,d\pi^2 + C_2 \intl_{S_2} u_{d,2} \,d\pi^2\\
&\le C_1 t(P_2,d)+ C_2 t(P_1,d),
\end{align*}
which is finite.

For $d=3$ and $k=4$, the graph $C_4$ does not satisfy the $3$-sparsity
condition, and indeed, as we have seen, the ortho-homomorphism measure has no
natural definition. Formula \eqref{EQ:BIPART} applies but the integral is
infinite.
\end{proof}

These computations imply that for $k\ge 5$, $u_{d,k}$ is a continuous function
on $S^{d-1}\times S^{d-1}$, so $u_{d,k}(x,x)$ is well defined, and
\begin{equation}\label{EQ:CKD-INT}
t(C_k,d) = \int_{S^{d-1}} u_{d,k}(x,x)\,d\pi(x).
\end{equation}
More explicit formulas for these densities will be given in Section
\ref{SEC:SPECTRAL} based on the spectrum of the graphop $\Ab$.

\subsection{Crowns}

For $n\ge 4$, we define the {\it $n$-crown} $\text{\rm Cr}_n$ as the bipartite
graph with bipartition $U\cup W$, where $U=\{u_1,\dots,u_n\}$,
$W=\{w_1,\dots,w_n\}$, and $w_i$ is connected to $u_{i-1}$, $u_{i}$, and
$u_{i+1}$ (subscripts modulo $n$). The $4$-crown is the skeleton of the
$3$-dimensional cube. For odd $n$, the $n$-crown is also known as the
``M\"obius ladder'', for even $n$, as the ``prism''. The $n$-crown is
$4$-sparse if $n\ge 4$.

\begin{prop}\label{PROP:3PATH}
{\rm(a)} If $d\ge4$, $n\ge 4$ and $(d,n)\notin \{(4,4), (4,5), (4,6), (5,4)\}$,
then $t(\text{\rm Cr}_n,d)$ is finite. {\rm(b)} $t(\text{\rm Cr}_4,4)$ is
infinite.
\end{prop}

\begin{proof}
(a) Let $x_i=x_{u_i}$ be independent random points of $S^{d-1}$. Let $\alpha_i$
be the angle between $x_i$ and $x_{i+1}$, and let $\vartheta_i$ be the angle
between the planes $\lin(x_{i-1},x_i)$ and $\lin(x_i,x_{i+1})$. Let
$Y_i=\sin\alpha_i = \Det(x_i,x_{i+1})$, $Z_i=\sin\vartheta_i$,
$D_i=\Det(x_{i-1},x_i,x_{i+1}) = Y_{i-1}Y_iZ_i$ and $W=Y_1^2\dots Y_n^2
Z_1\dots Z_n$. Then
\[
t(G,d) =C \E(W^{-1}),
\]
where the constant $C$ is computable by \eqref{EQ:BIPART}, but we don't need
this here. Let $B_i$ denote the event that $D_iD_{i+1}\ge D_jD_{j+1}$ for all
$j=1,\dots,n$. Clearly $\Pr(B_i)=1/n$ and $\E(W^{-1}~|~B_i)$ is independent of
$i$, which implies that $\E(W^{-1}~|~B_i)=\E(W^{-1})$.

Assume that $B_n$ occurs, and let $W_0=D_2\cdots D_{n-1}$. Then $D_nD_1 \ge
W^{2/n}$, and hence $W_0 = W/(D_1D_2) \le W^{(n-2)/n}$. Thus $W\ge
W_0^{n/(n-2)}$. Hence
\begin{align*}
\E(W^{-1}) &= \E(W^{-1}~|~B_n) \le \E\Big(W_0^{-n/(n-2)}~|~B_n\Big)
= \frac{\E\Big(W_0^{-n/(n-2)}\one_{B_n}\Big)}{\Pr(B_n)}\\
&=n\E\Big(W_0^{-n/(n-2)}\one_{B_n}\Big)\le n\E\Big(W_0^{-n/(n-2)}\Big).
\end{align*}
The advantage of considering $W_0$ is that we can write it as
\[
W_0 = Y_1 Y_2^2 \cdots  Y_{n-2}^2 Y_{n-1} Z_2\cdots Z_{n-1},
\]
and here all of the factors are independent random variables. So the
expectation of $W_0^{-n/(n-2)}$ is finite if and only if $\E(Y_1^{-n/(n-2)})$,
$\E(Y_i^{-2n/(n-2)})$ and $\E(Z_i^{-n/(n-2)})$ are finite. Clearly, finiteness
of the second expectation implies finiteness of the first one.

The expectations of powers of $Y_1$ and $Z_1$ are easy to compute: the density
function of (say) $\alpha_1$ is proportional to $(\sin\alpha_1)^{d-2}$, and so
\[
\E(Y_1^{-2n/(n-2)})=\frac{\intl_0^{\pi} (\sin\alpha)^{d-2-2n/(n-2)}\,d\alpha}{\intl_0^{\pi} (\sin\alpha)^{d-2}\,d\alpha}.
\]
The integral in the numerator is finite if the exponent of $\sin\alpha$ is
larger than $-1$; this means that
\begin{equation}\label{EQ:Y1}
d-2-2n/(n-2)>-1.
\end{equation}
Similarly, the density function of $Z_1$ is proportional to
$(\sin\vartheta_1)^{d-3}$ (the density of the angle between two random points
on the equator), and hence
\[
\E(Z_1^{-n/(n-2)})=\frac{\intl_0^{\pi} (\sin\alpha)^{d-3-n/(n-2)}\,d\alpha}{\intl_0^{\pi} (\sin\alpha)^{d-2}\,d\alpha}.
\]
As before, this is finite if and only if $d-3-n/(n-2)>-1$. It is not hard to
see that \eqref{EQ:Y1} is stronger. Rewriting \eqref{EQ:Y1} as $(d-3)(n-2)>4$,
we see that this holds for $d=4$ and $n\ge 7$, $d=5$ and $n\ge 5$ and $d\ge 6$,
$n\ge 4$.

\medskip

(b) For any three unit vectors $y_1,y_2,y_3$, we have
\[
|y_1\land y_3| + |y_2\land y_3|\ge |y_1\land y_2|.
\]
For vectors of arbitrary length, this gives
\[
|z_1\land z_3|\,|z_2| + |z_2\land z_3|\,|z_1| \ge |z_1\land z_2|\,|z_3|.
\]
Applying this to the vectors $z_i=x_i/x_4$, we get
\[
|x_1\land x_3\land x_4|\,|x_2\land x_4| + |x_2\land x_3\land x_4|\,|x_1\land x_4| \ge |x_1\land x_2\land x_4|\,|x_3\land x_4|.
\]
Setting $Y_i = |x_i\land x_3\land x_4|/|x_3\land x_4|$ (this is the distance of
$x_i$ from the plane $\lin(x_3,x_4)$, we get from this
\[
|x_1\land x_2\land x_4| \le Y_1 |x_2\land x_4| + Y_2\,|x_1\land x_4|
\le Y_1+Y_2.
\]
The same upper bound can be given on $|x_1\land x_2\land x_3|$, and trivially
\[
|x_i\land x_3\land x_4| = Y_i |x_3\land x_4| \le Y_i.
\]
The denominator in \eqref{EQ:BIPART} can be estimated as
\[
|x_1\land x_2\land x_3|\,|x_1\land x_2\land x_4|\,|x_1\land x_3\land x_4|\,|x_2\land x_3\land x_4|
\le (Y_1+Y_2)^2Y_1Y_2\le 4\max(Y_1,Y_2)^4.
\]
Note that the distributions of $Y_1$ and $Y_2$ do not depend on $x_3$ and
$x_4$, and so we can fix $L=\lin(x_3,x_4)$. Similarly as in the proof of Lemma
\ref{LEM:XEXPECT}, let $\vartheta_i$ be the angle between $x_i$ and $L$
($0\le\vartheta\le\pi/2$), then $Y_i=\sin\vartheta_i$. For a fixed $\vartheta$,
points at this distance $\sin\vartheta$ from $L$ form the direct product of two
circles $L\cap (\cos\vartheta)S^1$ and $L^\perp \cap (\sin\vartheta)S^1$, and
so their density is proportional to $\cos\vartheta\sin\vartheta$. Hence
$\E\big(\max(Y_1,Y_2)^{-4}\big)$ is proportional to
\begin{align*}
&\intl_{[0,\pi/2]^2}\frac{\sin\vartheta_1\cos\vartheta_1\sin\vartheta_2\cos\vartheta_2}{\max(\sin\vartheta_1,\sin\vartheta_2)^4}
\,d\vartheta_2\,d\vartheta_1 = 2 \intl_{\vartheta_2\le \vartheta_1}
\frac{\cos\vartheta_1\sin\vartheta_2\cos\vartheta_2}{(\sin\vartheta_1)^3} \,d\vartheta_2\,d\vartheta_1\\
&= \intl_{[0,\pi/2]} \frac{\cos\vartheta_1(\sin\vartheta_1)^2}{(\sin\vartheta_1)^3} \,d\vartheta_1
=\intl_{[0,\pi/2]} \frac{\cos\vartheta_1}{\sin\vartheta_1} \,d\vartheta_1,
\end{align*}
which is infinite.
\end{proof}

\section{Spectral formulas}\label{SEC:SPECTRAL}

\subsection{Powers of the graphop}

Let $\HH_d$ denote the function space $L^2(S^{d-1},\pi)$ where $\pi$ is the
uniform measure on $S^{d-1}$. If $Q$ is an element in the orthogonal group
$\Ort(d)$, then it also acts naturally on $\HH_d$ by $(fQ)(x)=f(Q(x))$ where
$f\in\HH_d$ and $x\in S^{d-1}$. We say that a linear operator $\Tb$ on $\HH_d$
is {\it rotation invariant}, if $Q\Tb Q^{-1}=\Tb$ holds for every $Q\in
\Ort(d)$.

Under the general correspondence between measures and linear operators, we can
define a bounded linear operator $\Ab=\Ab_d:~\HH_d\to \HH_d$ by letting
$(\Ab_df)(x)$ be the average of $f$ on the $(d-2)$-subsphere $x^\perp\cap
S^{d-1}$. (It is not hard to see that this is well-defined for almost all $x$;
see \cite{BackSz}.)

This operator satisfies
\begin{equation}\label{EQ:GG-DEF}
\langle  \Ab f, g\rangle= \intl_{S^{d-1}\times S^{d-1}} f(x)g(y)\,d\eta(x,y).
\end{equation}
for every $f\in L_p(S^{d-1},\pi)$ and $g\in L_q(S^{d-1},\pi)$ (see
\cite{BackSz}). This implies that it is self-adjoint. It is trivial that $\Ab$
is {\it monotone}: if $f\ge0$ then $\Ab f\ge 0$. We also note that $\Ab$ is
$1$-regular: $\Ab\one_{S^{d-1}}=\one_{S^{d-1}}$. This operator also has the
geometric property that it is {\it rotation invariant}, i.e.,

We say that an operator $\Tb:~L_2(S^{d-1})\to L_2(S^{d-1})$ is {\it
represented} by a measurable function $u:~S^{d-1}\to\R$ if
\[
(\Tb f)(x) = \intl_{S^{d-1}} u(x,y)f(y)\,d\pi(y).
\]
Then $\Tb$ is a Hilbert--Schmidt integral operator. The operator $\Ab$ cannot
be represented by any function, but its higher powers can: the operator $\Ab^k$
is represented by the function $u_{d,k}$ for all $k\ge 2$.

\subsection{Spherical harmonics}\label{SSEC:CYCLE-SPEC}

In this section we study the spectrum of the orthogonality operator $\Ab=\Ab_d$
for a fixed $d\ge 3$. As an application, we obtain formulas for the
ortho-homomorphism densities $t(C_k,d)$ in the next section.

The fact that $\Ab^k$ is a Hilbert--Schmidt operator for $k\ge2$ implies that
$\Ab$ is a compact operator. Let $\lambda_{n}$ $(n=0,1,2,\dots)$ be the
distinct nonzero eigenvalues of $\Ab$, and let $\Tb_n$ denote the orthogonal
projection onto the eigensubspace $W_n$ belonging to $\lambda_{n}$. Then the
expansion
\begin{equation}\label{EQ:SPECTRAL2}
\Ab^k = \sum_{n=0}^\infty\lambda_{n}^k \Tb_n
\end{equation}
is convergent in operator norm for $d=3$ and $k\ge 5$. Our goal is to give more
explicit formulas for $\lambda_n$ and $W_n$.

It is well-known that the action of the orthogonal group $\Ort(d)$ on the
Hilbert space $\HH_d=L^2(S^{d-1})$ has a unique decomposition into distinct
irreducible representations. These representations are carried by subspaces
$W_0,W_1,W_2,\dots$ of $L^2(S^{d-1})$, where $W_n$ consists of polynomials of
degree $n$ and has dimension
\begin{equation}\label{EQ:DIMW}
\dim(W_n)=\binom{d+n-1}{d-1}-\binom{d+n-3}{d-1}.
\end{equation}
Since the operator $\Ab$ is rotation invariant, standard arguments show that
each eigenspace of $\Ab$ is invariant under $\Ort(d)$. Hence each $W_n$ is
contained in one of the eigenspaces of $\Ab$ and thus elements in $W_n$ are
eigenvectors of $\Ab$ with identical eigenvalue $\lambda_n$.

The Gegenbauer polynomials $C_n^{(\alpha)}(x)$ (also called ultraspherical
polynomials) are orthogonal polynomials on $[-1,1]$ with respect to the weight
function $(1-x^2)^{\alpha-1/2}$. (In particular, $C_n^{(1/2)}(x)$ is the $n$-th
Legendre polynomial.) The significance of these polynomials for us is that if
$\alpha=d/2-1$, $n\in\mathbb{N}$ and $y\in\R^d$ is a fixed unit vector, then
the function $x\mapsto C_n^{(\alpha)}(x\cdot y)$ defined for $x\in S^{d-1}$ is
an eigenfunction of the operator $\Ab$ (called a {\it zonal spherical harmonic
function}). Furthermore, the corresponding eigenvalues (with appropriate
multiplicities) describe all the eigenvalues of $\Ab$.

It is not hard to calculate the eigenvalues corresponding to these functions.
It is clear that $f_n(y)=C_n^{(\alpha)}(1)$ and that $(\Ab
f_n)(y)=C_n^{(\alpha)}(0)$, and so the eigenvalue is
$C_n^{(\alpha)}(0)/C_n^{(\alpha)}(1)$. Fortunately, these special values of the
Gegenbauer polynomials are easily derived from the classical series expansion
\cite{Geg}
\begin{equation}\label{EQ:GEGEN2}
\sum_{n=0}^\infty C_n^{(\alpha)}(x)y^n = \frac{1}{(1-2xy+y^2)^{\alpha}}.
\end{equation}
Substituting $x=0$ and $x=1$, we get
\begin{equation}\label{EQ:GEGEN0}
C_n^{(\alpha)}(0)=
  \begin{cases}
    \displaystyle(-1)^r\binom{r+\alpha-1}{r}, & \text{if $n=2r$ is even}, \\
    0, & \text{if $n$ is odd},
  \end{cases}
\end{equation}
and
\begin{equation}\label{EQ:GEGEN1}
C_n^{(\alpha)}(1)=\binom{n+2\alpha-1}{n}.
\end{equation}

In our case when $\alpha=d/2-1$, both quantities are rational numbers. From
these formulas we obtain that the eigenvalue $\lambda_{n}$ of $\Ab$
corresponding to $n$-th zonal harmonic function is
\begin{equation}\label{EQ:LAMBDA-N}
\lambda_{n}=\frac{C_n^{(d/2-1)}(0)}{C_n^{(d/2-1)}(1)}=
  \begin{cases}
    \displaystyle(-1)^{r}\frac{(d-3)!!\,(2r-1)!!}{(2r+d-3)!!}, & \text{if $n=2r$ is even}, \\
    0, & \text{if $n$ is odd}.
  \end{cases}
\end{equation}
For even $d$, the formula for $\lambda_{n}$ can be simplified:
\begin{equation}\label{EQ:LAMBDA-2R}
\lambda_{n}=
  \begin{cases}
    \displaystyle (-1)^{n/2}\frac{(d-3)!!}{(n+1)(n+3)\dots(n+d-3)}, & \text{if $n$ is even}, \\
    0, & \text{if $n$ is odd}.
  \end{cases}
\end{equation}
Note that in this case the numerator is a constant (we consider $d$ fixed), and
the denominator is a polynomial in $n$. If $d=4$, then
$\lambda_{n}=(-1)^{n/2}/(n+1)$ for even $n$ and $\lambda_{n}=0$ for odd $n$.

The projections $\Tb_n$ to these subspaces can be described as well. The fact
that $W_n$ is finite dimensional implies that $\Tb_n$ is an integral kernel
operator representable by some measurable function $Q_n:~S^{d-1}\times
S^{d-1}\to\R$. Furthermore, since $W_n$ is an eigenspace of $\Ab$, each $Q_n$
is invariant under the natural action of the orthogonal group $\Ort(d)$. This
implies $Q_n(x,y)$ depends on the scalar product of $x$ and $y$ only. In other
words, there is a measurable function $f_n:~[-1,1]\to\R$ such that
$Q_n(x,y)=f_n(x\cdot y)$. This also means the for every fixed $y\in S^{d-1}$,
the map $x\mapsto f_n(x\cdot y)$ is in $W_n$ and thus these functions are the
zonal spherical harmonic functions. We obtain that
$f_n(x)=C_n^{(\alpha)}(x)c_n$ for some constant $c_n$ where $\alpha=(d-2)/2$.
The fact that $Q_n$ represents a projection onto $W_n$ implies that
$\dim(W_n)=\mathbb{E}_x Q_n(x,x)=C_n^{(\alpha)}(1)c_n$ and thus
$c_n=\dim(W_n)/C_n^{(\alpha)}(1)$. So
\[
f_n(x)=\dim(W_n)C_n^{(\alpha)}(x)/C_n^{(\alpha)}(1),
\]
and for $x\in S^{d-1}$ and $f\in\HH_d$,
\begin{equation}\label{EQ:TBN}
(\Tb_n f)(x) = \int_{S^{d-1}} f_n(x\cdot y)f(y)\,d\pi(y).
\end{equation}
We can apply \eqref{EQ:SPECTRAL2} (formally) for $k=1$:
\begin{equation}\label{EQ:AFX}
\Ab f(x) = \sum_{n=0}^\infty \lambda_{n}(\Tb_n f)(x)
=  \sum_{n=0}^\infty \lambda_{n} \int_{S^{d-1}} f_n(x\cdot y)f(y)\,d\pi(y).
\end{equation}
It is not clear when this infinite sum converges.

\subsection{Cycle densities}

We start with the expansion
\begin{equation}\label{EQ:CYCLE-SPEC}
t(C_k,d)= \tr(\Ab_d^k)=\sum_{n=0}^\infty \lambda_{n}^k\dim(W_n),
\end{equation}
convergent for $d=3$ and $k\ge 5$, and for $d\ge 4$ and $k\ge 4$. Substituting
values computed above, we get the formula
\begin{align}\label{EQ:T-CYCLE}
t(C_k,d) = \sum_{r=0}^\infty  &\left(\binom{d+2r-1}{d-1}-\binom{d+2r-3}{d-1}\right)
\left((-1)^{r}\frac{(d-3)!! (2r-1)!!}{(2r+d-3)!!}\right)^k.
\end{align}

If $d=4$, we obtain much nicer formulas:
\[
t(C_k,4)=\sum_{r=0}^\infty (2r+1)^{2-k}=\zeta(k-2)(1-2^{2-k})
\]
if $k$ is even and
\[
t(C_k,4)=\sum_{r=0}^\infty (2r+1)^{2-k}(-1)^r
\]
if $k$ is odd. In particular, $t(C_4,4)=\pi^2/8$. The case of a triangle is
interesting: the formula specializes to
\[
t(C_3,4)=1-\frac13+\frac15-\dots =\frac{\pi}{4}.
\]
The series is not absolute convergent, and we have no good argument to justify
the order in which it is summed; but the computations in Example
\ref{EXA:TRIANGLE} show that this is the ``right'' order.

If $d=3$, then the eigenvalues of $\Ab$ are $(-1/4)^r\binom{2r}{r}$ with
multiplicity $4r+1$ for $r=0,1,2,\dots$. This leads to
\[
t(C_k,3)=\sum_{r=0}^\infty (4r+1)(-1/4)^{rk}\binom{2r}{r}^k.
\]
Comparing with \eqref{EQ:TC6}, we get the identity
\begin{equation}\label{EQ:C6-ID}
\int\limits_{[0,\pi]^3} \frac{d\alpha\,d\beta\,d\theta}{\sqrt{1-\big(\cos\alpha\cos\beta
- \sin\alpha\sin\beta\cos\theta\big)^2}}
= \frac{\pi^4}{2} \sum_{r=0}^\infty (4r+1)4^{-6r} \binom{2r}{r}^6.
\end{equation}

\section{Approximations by graphons and graphs}

In this chapter we explain how the orthogonality graphs $H_d$ and
ortho-homomorphism densities fit into graph limit theory. Our goal is to find
sequences of graphons and finite graphs which approximate $H_d$ (or more
precisely the operator $\Ab_d$) in the sense that ortho-homomorphism densities
become limits of classical subgraph densities. As a consequence we obtain that
ortho-homomorphism densities behave a lot like subgraph densities. They satisfy
a variety of inequalities that are known in the graph theoretic framework. A
very interesting example is Sidorenko's conjecture, which has been proved for
quite a few classes of graphs. The ortho-homomorphism version of this
conjecture is especially nice: It says that the ortho-homomorphism density of
any bipartite graph in $H_d$ (for $d\geq 3$) is at least $1$. Our results in
this chapter will imply this for every bipartite graph that satisfies the
finite version of the conjecture.

For $x\in S^{d-1}$ and $0<r<\pi$, let $B_r(x)$ be the set of points $y \in
S^{d-1}$ such that $d(x,y)<r$ (where $d$ is the spherical distance), and let
$V_r=\pi(B_r(x))$. Let $f_{x,r}=\one_{B_{x,r}}$ be the indicator function of
$B_{r,x}$.

\begin{lemma}\label{invcomm}
If $A$ is a rotation invariant bounded operator on $\HH$, then $A$ is
self-adjoint, i.e., $A^*=A$. Any two rotation invariant bounded operators
commute.
\end{lemma}

\begin{proof}
For any two points $x,y\in S^{d-1}$ and $0<r<\pi$, there is a reflection $R\in
\Ort(d)$ in a hyperplane such that $R(x)=y$ and $R(y)=x$. For this reflection
we have $f_{x,r}=f_{y,r}R$ and $f_{y,r}=f_{x,r}R$ for every $r>0$. It follows
that
\begin{equation}\label{locsymm}
\langle f_{x,r}A,f_{y,r}\rangle=\langle f_{x,r}RAR,f_{y,r}\rangle
=\langle f_{x,r}RA,f_{y,r}R\rangle=\langle f_{y,r}A,f_{x,r}\rangle.
\end{equation}
For $r>0$ let $K_r:=\{f_{x,r}:x\in S^{d-1}\}$ and let $W_r$ denote the space of
finite linear combinations of elements in $K_r$. From equation (\ref{locsymm})
and the bilinearity of the scalar product, we obtain that $\langle
fA,g\rangle=\langle f,gA\rangle$ holds for any two functions $f,g\in W_r$.

Now let $f,g$ be arbitrary functions in $\HH_d$. It is easy to see that for
every $\epsilon>0$ there is an $r>0$ and two functions $f',g'\in W_r$ such that
$\|f-f'\|_2<\epsilon$, $\|g-g'\|_2<\epsilon$. Then
\[
|\langle f,gA\rangle-\langle f',g' A\rangle\leq |\langle f-f',gA\rangle|+|\langle f',(g-g')A\rangle|
< \epsilon\|g\|_2\|A\|_2+(\|f\|_2+\epsilon)\epsilon\|A\|_2
\]
and similarly $|\langle fA,g \rangle-\langle f'A,g'\rangle| <
\epsilon\|A\|_2(\|f\|_2+\|g\|_2+\epsilon)$. From $\langle f',g'A\rangle
=\langle f'A,g'\rangle$ and with $\epsilon\to 0$ we obtain that $\langle
f,gA\rangle =\langle fA,g\rangle$, showing that $A^*=A$.

To show the second claim, let $A,B$ be bounded rotation invariant operators.
Then $AB$ is also rotation invariant and so $AB=(AB)^*=B^*A^*=BA$ using the
first statement.
\end{proof}

Next we introduce a set of operators $\Mb_r$ on $\HH_d$ defined by
\[
(\Mb_r f)(x) = \frac1{V_r}\intl_{B_r(x)} f(y)\,d\pi(y).
\]
It is clear that $\Mb_r$ is rotation invariant, so by Lemma \ref{invcomm}
$\Mb_r$ and $\Ab_d$ commute and the product $\Cb_r:=\Ab_d\Mb_r$ is a
self-adjoint operator on $\HH_d$.

The operator $\Cb_r$ is a Hilbert--Schmidt operator with a nonnegative,
symmetric, bounded, measurable kernel $W_r:~S^{d-1}\times S^{d-1}\to\R$:
\[
(\Cb_r f)(x) = \intl_{S^{d-1}} W_r(x,y)f(y)\,d\pi(y).
\]
It is easy to see that for a fixed $x$, $W_r(x,y)$ is the density function of
the random point $y$ obtained by moving from $x$ in a random direction by
$\pi/2$ to a point $x'$, and then moving to a uniform random point of
$B_r(x')$.

Clearly $\int_{S^{d-1}} W_r(x,y)\,d\pi(y)=1$ for every $x$, so $W_r$ is a
$1$-regular graphon, and $t(H,W_r)$ is well-defined by \eqref{EQ:WG-DENS}. Our
main goal in this chapter is to prove that for a class of bipartite graphs $H$
we have
\begin{equation}
t(H,d)=\lim_{r\to 0} t(H,W_r)
\end{equation}
Our main tool is a rather explicit formula for the value of $t(H,W_r)$.

\begin{lemma}\label{applem1} For every $d\geq 3$ and $n\in\mathbb{N}$,
\begin{equation}\label{appeq1}
\intl_{S^{d-1}} \prod_{i=1}^n W_r(z,x_i)~d\pi(z)
=\frac{A_{d-1}^{n-1}A_{d-n-1}}{A_{d-2}^n}D_r(x_1,x_2,\dots,x_n).
\end{equation}
\end{lemma}

\begin{proof}
Let $z$ be a uniform random point on $S^{d-1}$ and let $z_1,z_2,\dots,z_n$ be
independent uniform elements on $S^{d-1}$ orthogonal to $z$. Let $x_i$ be
chosen uniformly from $B_r(z_i)$. By \eqref{EQ:KSTEP-GRAPHON}, the density
function of the joint distribution of $(x_1,\dots,x_n)$ is just the function on
the right hand side of \eqref{appeq1}. On the other hand, by \eqref{EQ:SNSD}
the joint distribution of $(z_1,\dots,z_n)$ has density function
\begin{equation}\label{appeq2}
\frac{A_{d-1}^{n-1}A_{d-n-1}}{A_{d-2}^n}D(z_1,x_2,\dots,z_n).
\end{equation}
Since $(x_1,\dots,x_n)$ is a random point in $B_r(z_1)\times\dots\times
B_r(z_n)$, the density function of $(x_1,\dots,x_n)$ is the average of
\eqref{appeq2} on $B_r(z_1)\times \dots\times B_r(z_n)$.
\end{proof}

\begin{lemma}\label{applem2}
Let $G=(V,E)$ be a $d$-sparse bipartite graph $(d\geq 3)$. Then
\begin{equation}\label{appeq3}
t(G,W_r)= \frac{A_{d-1}^{|E|-|W|}}{A_{d-2}^{|E|}}
\intl_{(S^{d-1})^U} \prod_{v\in W}A_{d-d(v)-1} D_r(x(N(v)))\,d\pi^U(x).
\end{equation}
\end{lemma}

\begin{proof}
Let $U\cup W$ be a bipartition of $V$, then using \eqref{appeq1},
\begin{align*}
t(G,W_r)& = \intl_{(S^{d-1})^V} \prod_{i\in W, j\in N(i)} W(x_i,x_j)\,d\pi^V(x) \\
&=\intl_{(S^{d-1})^U} \prod_{i\in W} \left( ~\intl_{S^{d-1}} \prod_{j\in N(i)}  W(x_i,x_j)\,d\pi(x_j)\right)\,d\pi^U\\
&= \intl_{(S^{d-1})^U} \prod_{i\in W} \left( \frac{A_{d-1}^{n-1}A_{d-n-1}}{A_{d-2}^n}D_r(x_1,x_2,\dots,x_n) \right)\,d\pi^U.
\end{align*}
Simplifying, we get \eqref{appeq3}.
\end{proof}

\begin{theorem}\label{appthm}
If $G$ is a bipartite graph that satisfies the sparsity condition, and
$t(G,d)<\infty$, then
\[
t(G,d)=\lim_{r\to 0} t(G,W_r).
\]
\end{theorem}

\begin{proof}
According to Lemma \ref{applem2} and the formula \eqref{EQ:BIPART}, it is
enough to prove that
\begin{equation}\label{appeq4}
\lim_{r\to 0}~~ \int\limits_{(S^{d-1})^U}\prod_{v\in W}D_r(x(N(v)))\,d\pi^U(x)
=\int\limits_{(S^{d-1})^U}\prod_{v\in W}D(x(N(v)))\,d\pi^U(x).
\end{equation}
Let
\[
\wh{D}_r(x)=\prod_{v\in W}D_r(x(N(v)))~~~{\rm and}~~~\wh{D}(x)=\prod_{v\in W}D(x(N(v))).
\]
It is clear that $\wh{D}_r(x)\to\wh{D}(x)$ as $r\to 0$ for almost all $x\in
(S^{d-1})^U$. By Lemma \ref{LEM:DETAVERAGE} we have that $\wh{D}_r(x)\leq C_d
\wh{D}(x)$ for some $c>0$ independent from $r$ and $x$. Since $t(G,d)$ is
finite, the function $\wh{D}$ is integrable, and so $c\wh{D}$ is an integrable
upper bound on $\wh{D}_r$. Thus \eqref{appeq4} follows by Lebesgue's Dominated
Convergence Theorem.
\end{proof}

The next theorem is a corollary of Theorem \ref{appthm}.

\begin{theorem}\label{appthm2}
For every $d\geq 3$ there is a sequence of finite graphs $\{G_i\}_{i=1}^\infty$
such that if a finite bipartite graph $H$ satisfies the sparsity condition,
then
\[
\lim_{i\to\infty}  t(H,G_i)/t(e,G_i)^{|E(H)|}=t(H,d).
\]
\end{theorem}

\begin{proof}
For $n\in\mathbb{N}$ let $U_n:=W_{d,1/n}/\|W_{d,1/n}\|_\infty$, then $U_n$ is a
symmetric measurable function with values in $[0,1]$. It follows from the
results in \cite{LSzLimit} that there is a finite graph $G_n$ such that
\[
\Bigl|t(H,G_n)/t(e,G_n)^{|E(H)|}-t(H,U_n)/t(e,U_n)^{|E(H)|}\Bigr|\leq 1/n.
\]
Since $t(e,W_{d,1/n})=1$, we also have that
$t(H,U_n)/t(e,U_n)^{|E(H)|}=t(H,W_{d,1/n})$. Together with Theorem
\ref{appthm}, this completes the proof.
\end{proof}

Theorem \ref{appthm2} shows that in some sense $H_d$ is a limit of finite
graphs. It is interesting to mention that the sequence $\{G_i\}_{i=1}^\infty$
given by the proof of the theorem is a sparse graph sequence. We also have an
interesting corollary of Theorem \ref{appthm2}.

\begin{corollary}\label{appthmcor}
If $H$ is a $d$-sparse bipartite graph that satisfies Sidorenko's conjecture,
then $t(H,d)\geq 1$.
\end{corollary}

Sidorenko's conjecture is verified for large families of bipartite graphs, and
thus Corollary \ref{appthmcor} implies several non-trivial inequalities for
ortho-homomorphism densities. Some other graph theoretic inequalities can also
be transported to ortho-homomorphism densities with the help of Theorem
\ref{appthm2}; but we omit the details here.

\section{Open problems}

Let us conclude with some special and more general problems left open by our
work.

\begin{problem}
Decide the finiteness of $t(\text{\rm Cr}_n,d)$ the open cases $(d,n)\in
\{(4,5), (4,6), (5,4)\}$ in Proposition \ref{PROP:3PATH}.
\end{problem}

\begin{problem}
Characterize graphs $G$ and dimensions $d$ for which $t(G,d)$ is finite. As an
interesting example: if $G$ is the incidence graph of the Fano plane, is
$t(G,4)$ finite?
\end{problem}

\begin{problem}
The fact that the cube graph $\text{\rm Cr}_4$ is, in a sense, exceptional
among crowns, may be related to the fact that for $4$-sparse graphs, the real
algebraic variety of all ortho-homomorphisms in dimension $4$ is irreducible,
except for the cube. Is there a more substantial connection?
\end{problem}

\begin{problem}
Make sense of the identity \eqref{EQ:C6-ID}, perhaps generalized to all cycles
and all dimensions.
\end{problem}

\begin{problem}
Let $G$ be a $d$-sparse graph, and let $\mu$ be a probability measure on
$\Sigma_{G,d}$ with Markovian conditioning. Is $\mu$ uniquely determined by
$G$?
\end{problem}

\begin{problem}
Are there natural graph sequences converging to the orthogonality graph? The
orthogonality graph $H_{p,d}$ of $\Fbb_p^d\setminus\{0\}$ (more exactly, the
conjugacy graph in the projective space $\Pbb_p^{d-1}$) is a natural example,
but it does not work: cf.~\cite{BackSz}, Section 12.5, from which it follows
that conjugacy graphs of finite projective spaces tend to a trivial limit in
the sense of action convergence (a form of right convergence). From the other
side, it is easy to compute that $t^*(K_3,H_{p,3})=(p^2+p+1)/(p+1)^2\sim 1$,
while we have seen that $t(K_3,H_3)=2/\pi$, showing that $H_{p,3}$ does not
tend to $H_3$ in the local sense either.
\end{problem}

\begin{problem}
Instead of random unit vectors, we could consider other probability
distributions; Gaussian would be a natural choice. In the sequentially
constructed random map, we map each node $v$ onto a random vector from the
standard Gaussian distribution on the subspace orthogonal to the previously
chosen images of neighbors of $v$. We expect that a density function making
this mapping independent of the order of the nodes can be constructed along the
same lines as in this paper. This construction may have even nicer properties
than our random ortho-homomorphism; but this is not discussed in this paper.

As another natural generalization, we could determine subgraph densities in the
uniform measure on pairs of points of a unit sphere at any given distance
(different from $\pi/2$). Even more generally, perhaps the methods above can be
applied to any probability measure on pairs of points in $\R^d$ invariant under
the orthogonal group.
\end{problem}

\begin{problem}
Based on \eqref{EQ:TBN} and \eqref{EQ:AFX}, one can (formally) derive the
following formula:
\begin{align}\label{TGA}
t(G,d) = \sum_{\tau:\,E(H)\to\Nbb} \intl_{(S^{d-1})^V}
\prod_{ij\in E} \lambda_{\tau(ij)} f_{\tau(ij)}(x_i\cdot x_j)\,d\pi^V(x).
\end{align}
Note that the product in the formula is a multivariate polynomial on $\R^{dn}$
with rational coefficients which depends on the edge labeling $\tau$. It is not
clear when this infinite sum converges and when the equality holds.
\end{problem}

\def\RSA{{\it Random Struc.\ Alg.} }
\def\CCA{{\it Combinatorica} }
\def\JCTB{{\it J.~Combin.\ Theory B} }
\def\JCTA{{\it J.~Combin.\ Theory A} }
\def\CPC{{\it Combin.\ Prob.\ Comput.} }
\def\EJC{{\it Europ.\ J.~Combin.} }
\def\ELJC{{\it Electr.\ J.~Combin.} }
\def\GC{{\it Graphs and Combin.} }
\def\JGT{{\it J.~Graph Theory} }
\def\ADV{{\it Advances in Math.} }
\def\ADVA{{\it Advances in Applied Math.} }
\def\AMH{{\it Acta Math. Hung.} }
\def\GAFA{{\it Geom.\ Func.\ Anal.} }
\def\STOC#1 {{\it Proc.\ #1$^\text{th}$ ACM Symp.\
on Theory of Comput.} }
\def\FOCS#1 {{\it Proc.\ #1$^\text{th}$ Ann.\ IEEE
Symp.\ on Found.\ Comp.\ Science} }
\def\SODA#1 {{\it Proc.\ #1$^\text{th}$ Ann.\
ACM-SIAM Symp.\ on Discrete Algorithms} }
\def\DCG{{\it Discr.\ Comput.\ Geom.} }
\def\DM{{\it Discr.\ Math.} }
\def\DAM{{\it Discr.\ Applied Math.} }
\def\SJC{{\it SIAM J.~Comput.} }
\def\SDM{{\it SIAM J.~Discr.\ Math.} }
\def\TIT{{\it IEEE Trans.\ Inform.\ Theory} }
\def\LAA{{\it Linear Algebra Appl.} }

\end{document}